\documentclass[10pt, reqno, a4paper]{amsart}
\usepackage{amsfonts}
\usepackage{amsmath}
\usepackage{amsthm}
\usepackage{stmaryrd}

\usepackage{latexsym}
\usepackage{amssymb}
\usepackage{amsmath,amsthm}
\usepackage{mathrsfs}
\usepackage{graphicx}
\usepackage{verbatim}

\usepackage[all]{xy}
\usepackage{multirow}
\usepackage{placeins}

\def \b {\mathfrak{b}}
\def \f {\mathfrak{f}}
\def \g {\mathfrak{g}}
\def \h {\mathfrak{h}}

\def \SS {\mathfrak {S}}

\def\CC {{\mathbb C}}     
\def\ZZ {{\mathbb Z}}     

\newcommand{\gl}   {\mathrm{gl}}
\renewcommand{\sl} {\mathrm{sl}}
\newcommand{\so}   {\mathrm{so}}

\renewcommand{\sp}   {\mathrm{sp}}
\newcommand{\soc}  {\mathrm{soc}}

\newcommand{\codim}  {\mathrm{codim}}

\newcommand{\im} {\mathrm{im}}
\newcommand{\id} {\mathrm{id}}

\newcommand{\sign} {\mathrm{sign}}

\renewcommand {\phi} {\varphi}

\newtheorem{thm}{Theorem}[section]
\newtheorem{lemma}[thm]{Lemma}

\newtheorem{prop}[thm]{Proposition}
\newtheorem{coro}[thm]{Corollary}

\newtheorem{df}{Definition}[section]

\begin{document}

\title{Branching laws for tensor modules over classical locally finite Lie algebras}
\author{Elitza Hristova}
\date{}

\setcounter{page}{1}

\maketitle

\begin{abstract} 
Let $\g'$ and $\g$ be isomorphic to any two of the Lie algebras $\gl(\infty), \sl(\infty), \sp(\infty)$, and $\so(\infty)$. Let $M$ be a simple tensor $\g$-module (\cite{PSt}, \cite{PSe}). We introduce the notion of an embedding $\g' \subset \g$ of general tensor type and derive branching laws for triples $\g', \g, M$, where $\g' \subset \g$ is an embedding of general tensor type. 
More precisely, since $M$ is in general not semisimple as a $\g'$-module, we determine the socle filtration of $M$ over $\g'$. Due to the description of embeddings of classical locally finite Lie algebras given in \cite{DP}, 
our results hold for all possible embeddings $\g' \subset \g$ unless $\g' \cong \gl(\infty)$.
\end{abstract}

\section{Introduction}
Given an embedding $\g' \subset \g$ of two Lie algebras and a simple $\g$-module $M$, the branching problem is to determine the structure of $M$ as a $\g'$-module. This is a classical problem in the theory of finite-dimensional Lie algebras. By Weyl's semisimplicity theorem, when $\g'$ is finite dimensional semisimple the branching problem reduces to finding the multiplicity of any simple $\g'$-module $M'$ as a direct summand of $M$. This is however not a simple task, due to the abundance of possible isomorphism classes of embeddings $\g' \subset \g$ (\cite{Dy}). Therefore, even for the classical series of Lie algebras explicit solutions of the branching problem are known only for specific cases. Such solutions are referred to as \textit{branching laws} or \textit{branching rules} and examples can be found in e.g. \cite{Z}, \cite{HTW}.

In this paper we consider the branching problem for the classical locally finite Lie algebras. These are the Lie algebras $\gl(\infty), \sl(\infty), \sp(\infty)$, and $\so(\infty)$, and they are defined as unions of the respective finite-dimensional Lie algebras under the upper-left corner inclusions. Here the situation is quite different from the finite-dimensional case. On the one hand, the description of the Lie algebra embeddings given in \cite{DP} is much simpler than the classical description of Dynkin in the finite-dimensional case. 
On the other hand, the modules of interest, called simple tensor modules, are in general not completely reducible over the subalgebra. Therefore, the branching problem involves more than just determining the multiplicities of all simple constituents. One has to determine a semisimple filtration of the given module over the subalgebra and it is a natural choice to work with the socle filtration. In this way, the goal of the present work is to solve the following branching problem. Given an embedding $\g' \subset \g$ of two classical locally finite Lie algebras and a simple tensor $\g$-module $M$, find the socle filtration of $M$ as a $\g'$-module. 

The structure of the paper is as follows. We start by giving some background on locally finite Lie algebras and by presenting some finite-dimensional branching laws which are used in the paper. Following the description of embeddings of classical locally finite Lie algebras given in \cite{DP}, we introduce the notion of an embedding $\g' \subset \g$ of general tensor type. In the case $\g' \not\cong \gl(\infty)$ this notion describes all possible embeddings. One of the main results of the paper is Theorem \ref{thmMainCC} which shows that the branching problem for embeddings of general tensor type can be reduced to branching problems for embeddings of three simpler types. Then in Section \ref{secGL} we determine explicitly the branching laws for these three types of embeddings in the case when $\g', \g \cong \gl(\infty)$ and $M$ is any simple tensor $\g$-module (Theorems \ref{thmEmb1}, \ref{thmEmb2}, \ref{thmEmb3}). Since all other cases of embeddings follow the same ideas, we skip the proofs and list the end results in several tables in the Appendix. 

{\bf Acknowledgements.} I want to thank my thesis adviser Ivan Penkov for introducing me to the topic and for his support and guidance during my PhD work and during the writing of this text. I would like also to acknowledge the DAAD (German Academic Exchange Service) for 
giving me financial support during my PhD studies.

\section{Preliminaries}
\subsection{The classical locally finite Lie algebras}
The ground field is $\CC$. A countable-dimensional Lie algebra is called \textit{locally finite} if every finite subset of $\g$ is contained in a finite-dimensional subalgebra. Equivalently, $\g$ is locally finite if it admits an exhaustion $\g = \bigcup_{i \in \ZZ_{>0}}\g_i$ where
\begin{align*} 
\g_1 \subset \g_2 \subset \dots \subset \g_i \subset \dots
\end{align*}
is a sequence of nested finite-dimensional Lie algebras. 
The classical locally finite Lie algebras $\gl(\infty)$, $\sl(\infty)$, $\sp(\infty)$, and $\so(\infty)$ are defined respectively as $\gl(\infty) = \bigcup_{i \in \ZZ_{>0}}\gl(i)$, $\sl(\infty) = \bigcup_{i \in \ZZ_{>0}}\sl(i)$, $\sp(\infty) = \bigcup_{i \in \ZZ_{>0}}\sp(2i)$, and $\so(\infty) = \bigcup_{i \in \ZZ_{>0}}\so(i)$ via the natural inclusions $\gl(i) \subset \gl(i+1)$, $\sl(i) \subset \sl(i+1)$, $\sp(2i) \subset \sp(2i + 2)$, and $\so(i) \subset \so(i+1)$. 

Next, we give an equivalent definition of the above four Lie algebras. Let $V$ and $V_*$ be countable-dimensional vector spaces over $\CC$ and let $\left\langle \cdot,\cdot\right\rangle: V \times V_* \rightarrow \CC$ be a non-degenerate bilinear pairing. The vector space $V \otimes V_*$ is endowed with the structure of an associative algebra such that
$$
(v_1 \otimes w_1)(v_2 \otimes w_2) = \left\langle v_2, w_1 \right\rangle v_1\otimes w_2
$$
where $v_1, v_2 \in V$ and $w_1, w_2 \in V_*$. We denote by $\gl(V, V_*)$ the Lie algebra arising from the associative algebra $V \otimes V_*$, and by $\sl(V, V_*)$ we denote its commutator subalgebra $[\gl(V, V_*), \gl(V, V_*)]$. If $\left\langle \cdot,\cdot\right\rangle: V \times V \rightarrow \CC$ is an antisymmetric non-degenerate bilinear form, we define the Lie algebra $\gl(V, V)$ as above by taking $V_* = V$. In this case $S^2(V)$, the second symmetric power of $V$, is a Lie subalgebra of $\gl(V,V)$ and we denote it by $\sp(V)$. Similarly, if $\left\langle \cdot,\cdot\right\rangle: V \times V \rightarrow \CC$ is a symmetric non-degenerate bilinear form, we again define $\gl(V, V)$ by taking $V_* = V$ and then $\bigwedge^2(V)$ is a Lie subalgebra of $\gl(V,V)$, which we denote by $\so(V)$.

The vector spaces $V$ and $V_*$ are naturally modules over the Lie algebras defined above, such that $(v_1 \otimes w_1)\cdot v_2 =\left\langle v_2, w_1 \right\rangle v_1$ and $(v_2 \otimes w_2)\cdot w_1 = -\left\langle v_2, w_1 \right\rangle w_2$ for any $v_1, v_2 \in V$ and $w_1, w_2 \in V_*$. We call them respectively the \textit{natural} and the \textit{conatural representations}. In the cases of $\sp(V)$ and $\so(V)$ we have $V = V_*$. 

By a result of Mackey \cite{M}, there always exist dual bases $\{\xi_i\}_{i \in I}$ of $V$ and $\{\xi_i^*\}_{i \in I}$ of $V_*$ indexed by a countable set $I$, so that $\left\langle \xi_i, \xi_j^* \right\rangle = \delta_{ij}$. 
Using these bases, we can identify $\gl(V,V_*)$ with the Lie algebra $\gl(\infty)$. Similarly, $\sl(V,V_*) \cong \sl(\infty)$, $\sp(V) \cong \sp(\infty)$, and $\so(V) \cong \so(\infty)$. 

In what follows let $\g \cong \gl(\infty), \sl(\infty), \sp(\infty)$, or $\so(\infty)$. We set $V^{\otimes (p,q)} = V^{\otimes p} \otimes V_*^{\otimes q}$. If $\g \cong \sp(\infty), \so(\infty)$, we always consider $q= 0$. 
It is shown in \cite{PSt} that $V^{\otimes (p,q)}$ is in general not semisimple and its socle filtration is explicitly described. 

The \textit{socle filtration} of any module $M$ over an algebra is defined as
\begin{align*}
0 \subset \soc M \subset \soc^{(2)}M \subset \cdots \subset \soc^{(r)} M \subset \cdots
\end{align*}
where $\soc M$, called the \textit{socle} of $M$, is the maximal semisimple submodule of $M$, or equivalently the sum of all simple submodules of $M$. The other terms in the filtration are defined inductively as follows: $\soc^{(r+1)}M = \pi_r^{-1}(\soc(M/\soc^{(r)} M))$, where $\pi_r: M \rightarrow M/\soc^{(r)} M$ is the natural projection. The semisimple modules $\overline{\soc}^{(r+1)}M := \soc^{(r+1)}M / \soc^{(r)}M$ are called the \textit{layers} of the socle filtration. We say that $M$ has \textit{finite Loewy length} $l$ if the socle filtration of $M$ is finite and $l = \min \{r | \soc^{(r)}M = M\}$. 

The following three properties of socle filtrations will be very useful in the sequel.
\begin{itemize}
\item If $N \subseteq M$, then for all $r$
\begin{align} \label{eqPropSoc1}
\soc^{(r)} N = (\soc^{(r)}M) \cap N .
\end{align}
\item If $M_1$ and $M_2$ are modules over the same algebra and $M = M_1 \cap M_2$, then
\begin{align} \label{eqPropSoc2}
\soc^{(r)} M = (\soc^{(r)}M_1) \cap (\soc^{(r)} M_2).
\end{align}
\item If $M$ and $N$ are any two modules over the same algebra, then 
\begin{align} \label{eqPropSoc} 
\soc^{(r)} (M \oplus N) = \soc^{(r)} M \oplus \soc^{(r)} N.
\end{align}
\end{itemize}

In what follows, if $M$ is a module over the Lie algebra $\g$ we will use the notation $\soc^{(r)}_{\g} M$ instead of $\soc^{(r)} M$.

When $\g \cong \gl(\infty)$, we set $V^{\{p,q\}} = \soc_{\g} V^{\otimes (p,q)}$. It is also the maximal semisimple submodule of $V^{\otimes (p,q)}$ for $\g \cong \sl(\infty)$ (\cite{PSt}). For $\g \cong \sp(\infty), \so(\infty)$ the maximal semisimple submodule of $V^{\otimes p}$ is denoted respectively by $V^{\left\langle p \right\rangle}$ and $V^{[p]}$. 

By definition, a $\g$-module M is called a \textit{tensor module} if it is a subquotient of a finite direct sum of copies of $\bigoplus_{p+q \leq r} V^{\otimes (p,q)}$ for some integer $r$. 
If $M$ is simple, being a tensor module is equivalent to being a submodule of $V^{\otimes (p,q)}$ for some $p,q$ (\cite{PSt}). Moreover, it is shown in \cite{PSt} that there exists a choice of Borel subalgebra $\b$ in $\g$ such that all simple tensor modules are $\b$-highest weight modules.
For $\g \cong \gl(\infty), \sl(\infty)$  the simple tensor modules coincide. Their highest weights are given by pairs of non-negative integer partitions $(\lambda, \mu)$ and we denote them by $V_{\lambda, \mu}$. If $V_{\lambda, \mu} \subset V^{\otimes (p,q)}$, then $\vert \lambda \vert = p$ and $\vert \mu \vert  = q$. 
For $\g \cong \sp(\infty), \so(\infty)$ the simple tensor modules are denoted respectively by $V_{\left\langle \lambda \right\rangle}$ and $V_{[\lambda]}$, where $\lambda$ is a non-negative integer partition. If $V_{\left\langle \lambda \right\rangle} \subset V^{\otimes p}$ or respectively $V_{[\lambda]} \subset V^{\otimes p}$, then $\vert \lambda \vert = p$. The simple tensor modules are constructed explicitly in \cite{PSt} using a generalization of Weyl's construction for irreducible $\gl(n)$-modules (see \cite{FH}). 

Next, let $\g' \subset \g$ be an embedding of two classical locally finite Lie algebras. Then the following theorem holds.
\begin{thm} \cite{DP} \label{thmDP} Let $\g \cong \gl(\infty), \sl(\infty), \sp(\infty)$, or $\so(\infty)$ and let $\g'$ be a simple infinite-dimensional subalgebra of $\g$. Let $V$ and $V_*$ be respectively the natural and conatural represenations of $\g$. Similarly, let $V'$ and $V_*'$ be the natural and conatural representations of $\g'$. Then
\begin{equation} \label{eqGenForm}
\begin{aligned}
&\soc_{\g'}  V \cong kV' \oplus lV'_* \oplus N_a,  \quad &V/\soc_{\g'} V \cong N_b, \\
&\soc_{\g'}  V_* \cong lV' \oplus kV'_* \oplus N_c,  \quad &V_*/\soc_{\g'} V_* \cong N_d,
\end{aligned}
\end{equation}
where $k,l \in \ZZ_{\geq 0}$ such that at least one of them is in $\ZZ_{>0}$, and $N_a$, $N_b$, $N_c$, and $N_d$ are finite- or countable-dimensional trivial $\g'$-modules.
\end{thm}

We denote $a = \dim N_a$, $b= \dim N_b$, $c = \dim N_c$, and $d = \dim N_d$.

Motivated by Theorem \ref{thmDP} we give the following definition.

\begin{df} An embedding $\g' \subset \g$ of two classical locally finite Lie algebras is said to be of general tensor type if it satisfies property (\ref{eqGenForm}). 
\end{df}

In view of Theorem \ref{thmDP}, all possible embeddings $\g' \subset \g$ are of general tensor type under the assumption that $\g' \not\cong \gl(\infty)$.

It is proven in \cite{PSe} that if $\g' \subset \g$ is an embedding of simple classical locally finite Lie algebras and $M$ is a simple tensor $\g$-module, then $M$ has finite Loewy length as a $\g'$-module and all simple constituents in the socle filtration of $M$ as a $\g'$-module are simple tensor $\g'$-modules. It is not difficult to show that this holds moreover for any embedding of general tensor type of any two classical locally finite Lie algebras (Proposition \ref{propSLExclude}). This fact and Theorem \ref{thmDP} are one of the first motivations for considering the branching problem in the context of classical locally finite Lie algebras. 

\subsection{Some finite-dimensional branching laws} \label{secFinBran}
An embedding $\f_1 \subset \f_2$ of finite-dimensional classical Lie algebras is called \textit{diagonal} if
$$
{V_2}_{ \downarrow \f_1} \cong \underbrace{V_1 \oplus \dots \oplus V_1}_{l}\oplus
											\underbrace{V_1^* \oplus \dots \oplus V_1^*}_{r}\oplus
											\underbrace{N_1 \oplus \dots \oplus N_1}_{z}
$$	
where $V_i$ is the natural $\f_i$-module ($i=1,2$), $V_1^*$ is dual to $V_1$, and $N_1$ is the trivial one-dimensional $\f_1$-module. The triple $(l,r,z)$ is called the \textit{signature} of the embedding. 

\begin{prop}\cite{Z} [Gelfand-Tsetlin rule] \label{propGT}
Consider an embedding $\gl(n-1) \rightarrow \gl(n)$ of signature $(1,0,1)$. Let $\lambda = (\lambda_1, \dots, \lambda_n)$ be a non-negative integer partition and let $V^n_{\lambda}$ denote the irreducible highest weight $\gl(n)$-module with highest weight $\lambda$. Then
$$
{V^{n}_{\lambda}}_{\downarrow \gl(n-1)} \cong \bigoplus_{\sigma}V^{n-1}_{\sigma},
$$
where $\sigma = (\sigma_1, \dots, \sigma_{n-1})$ runs over all integer partitions which interlace $\lambda$, i.e. such that
$$
\lambda_1 \geq \sigma_1 \geq \lambda_2 \geq \sigma_2 \geq \dots \geq \lambda_{n-1} \geq \sigma_{n-1} \geq \lambda_n.
$$
\end{prop}

The Gelfand-Tsetlin rule can be iterated to obtain a branching law for embeddings of $\gl(n)$ into $\gl(n+k)$. 
\begin{prop} \label{propIterGT}
Let $\gl(n) \rightarrow \gl(n+k)$ be an embedding of signature $(1,0,k)$. 
Then
$$
{V^{n+k}_{\lambda}}_{\downarrow \gl(n)} \cong \bigoplus_{\sigma} m^k_{\lambda, \sigma} V^n_{\sigma},
$$
where the multiplicity $m^k_{\lambda, \sigma}$ is the number of possible sequences of weights 
\begin{align*}
&\lambda = (\lambda_1, \lambda_2, \dots, \lambda_{n+k-1}, \lambda_{n+k}),\\
&\sigma^1 = (\sigma^1_1, \sigma^1_2, \dots,  \sigma^1_{n+k-1}), \\
&\dots  \dots, \\
&\sigma^{k-1} = (\sigma^{k-1}_1, \dots, \sigma^{k-1}_{n+k-2}),\\
&\sigma = (\sigma_1, \dots, \sigma_n)
\end{align*}
such that each consecutive row interlaces the previous one.
\end{prop}

In what follows we will refer to the numbers $m^k_{\lambda, \sigma}$ as \textit{Gelfand-Tsetlin multiplicities}.

Next, we consider embeddings of signature $(k,0,0)$. Let $\lambda$ and $\mu$ be two non-negative integer partitions with $p$ and $q$ parts, where $p+q \leq n$.
Let $V^{n}_{\lambda, \mu}$ denote the irreducible $\gl(n)$-module with highest weight $(\lambda, \mu) = (\lambda_1, \dots \lambda_p, 0, \dots 0, -\mu_q, \dots, -\mu_1)$. We will derive the desired branching rule as an easy generalization of the results in \cite{HTW}. 

Consider first the block-diagonal subalgebra $\gl(n) \oplus \gl(m) \subset \gl(n+m)$. By Theorem 2.2.1 in \cite{HTW} we have the following decomposition:

\begin{align*}
{V^{n+m}_{\lambda, \mu}}_{\downarrow \gl(n)\oplus\gl(m)} \cong \bigoplus_{\substack{\alpha^+, \beta^+, \alpha^-, \beta^-}}
c^{(\lambda, \mu)}_{(\alpha^+, \alpha^-), (\beta^+, \beta^-)} V^n_{\alpha^+, \alpha^-} \otimes V^m_{\beta^+, \beta^-},
\end{align*} 
where
\begin{align*}
c^{(\lambda, \mu)}_{(\alpha^+, \alpha^-), (\beta^+, \beta^-)} = \sum_{\gamma^+, \gamma^-, \delta}
c^{\lambda}_{\gamma^+ \delta} c^{\mu}_{\gamma^- \delta} c^{\gamma^+}_{\alpha^+ \beta^+}c^{\gamma^-}_{\alpha^- \beta^-},
\end{align*}
and the numbers $c^{\alpha}_{\beta, \gamma}$ are Littlewood-Richardson coefficients. Next, we consider the direct sum of $k > 2$ copies of $\gl(n)$ as a subalgebra of $\gl(kn)$ by block-diagonal inclusion. Then one can iterate the above branching rule to obtain the following:
\begin{align} \label{eqBranch1}
{V^{kn}_{\lambda, \mu}}_{\downarrow \gl(n) \oplus \dots \oplus \gl(n)} \cong \bigoplus_{\substack{\beta^+_1, \dots, \beta^+_k \\ \beta^-_1, \dots, \beta^-_k}}
C^{(\lambda, \mu)}_{(\beta^+_1, \dots, \beta^+_k) (\beta^-_1, \dots, \beta^-_k)} V^n_{\beta_1^+, \beta_1^-} \otimes \dots \otimes V^n_{\beta_k^+, \beta_k^-},
\end{align}
where 
\begin{align*}
&C^{(\lambda, \mu)}_{(\beta^+_1, \dots, \beta^+_k) (\beta^-_1, \dots, \beta^-_k)} =
\sum_{\substack{\alpha^+_1, \dots \alpha^+_{k-2} \\ \alpha^-_1, \dots \alpha^-_{k-2}}} c^{(\lambda, \mu)}_{(\alpha^+_1, \alpha^-_1)(\beta^+_1, \beta^-_1)}c^{(\alpha^+_1, \alpha^-_1)}_{(\alpha^+_2, \alpha^-_2)(\beta^+_2, \beta^-_2)} \dots c^{(\alpha^+_{k-3}, \alpha^-_{k-3})}_{(\alpha^+_{k-2}, \alpha^-_{k-2})(\beta^+_{k-2}, \beta^-_{k-2})}c^{(\alpha^+_{k-2}, \alpha^-_{k-2})}_{(\beta^+_{k-1}, \beta^-_{k-1})(\beta^+_k, \beta^-_k)}.
\end{align*}

The next step is to consider the $\gl(n) \oplus \gl(n)$-module $V^n_{\alpha^+, \alpha^-} \otimes V^n_{\beta^+, \beta^-}$. By Theorem 2.1.1 in \cite{HTW}, as a module over the subalgebra $\gl(n) = \{x\oplus x | x \in \gl(n)\}$, $V^n_{\alpha^+, \alpha^-} \otimes V^n_{\beta^+, \beta^-}$ has the decomposition
\begin{align*}
{V^n_{\alpha^+, \alpha^-} \otimes V^n_{\beta^+, \beta^-}}_{\downarrow \gl(n)} \cong \bigoplus_{\lambda', \mu'}
d_{(\alpha^+, \alpha^-), (\beta^+, \beta^-)}^{(\lambda', \mu')} V^n_{\lambda', \mu'},
\end{align*} 
where
\begin{align*}
d_{(\alpha^+, \alpha^-), (\beta^+, \beta^-)}^{(\lambda', \mu')} = \sum_{\substack{\alpha_1, \alpha_2, \beta_1, \beta_2 \\ \gamma_1, \gamma_2}} c^{\alpha^+}_{\alpha_1 \gamma_1}c^{\beta^-}_{\gamma_1 \beta_2}c^{\alpha^-}_{\beta_1 \gamma_2}c^{\beta^+}_{\gamma_2 \alpha_2} c^{\lambda'}_{\alpha_2 \alpha_1}  c^{\mu'}_{\beta_2 \beta_1}.
\end{align*}

Iterating this branching rule, we obtain for $k > 2$
\begin{align} \label{eqBranch2}
{V^n_{\beta_1^+, \beta_1^-} \otimes \dots \otimes V^n_{\beta_k^+, \beta_k^-}}_{\downarrow \gl(n)} \cong \bigoplus_{\lambda', \mu'}
D^{(\lambda', \mu')}_{(\beta^+_1, \dots, \beta^+_k) (\beta^-_1, \dots, \beta^-_k)} V^n_{\lambda', \mu'},
\end{align}
where
\begin{align*}
&D^{(\lambda, \mu)}_{(\beta^+_1, \dots, \beta^+_k) (\beta^-_1, \dots, \beta^-_k)} =
\sum_{\substack{\alpha^+_1, \dots \alpha^+_{k-2} \\ \alpha^-_1, \dots \alpha^-_{k-2}}} d^{(\alpha^+_1, \alpha^-_1)}_{(\beta^+_1, \beta^-_1)(\beta^+_2, \beta^-_2)}d^{(\alpha^+_2, \alpha^-_2)}_{(\alpha^+_1, \alpha^-_1)(\beta^+_3, \beta^-_3)} \dots d^{(\alpha^+_{k-2}, \alpha^-_{k-2})}_{(\alpha^+_{k-3}, \alpha^-_{k-3})(\beta^+_{k-1}, \beta^-_{k-1})}d^{(\lambda, \mu)}_{(\alpha^+_{k-2}, \alpha^-_{k-2})(\beta^+_k, \beta^-_k)}.
\end{align*}

The coefficients $C^{(\lambda, \mu)}_{(\beta^+_1, \dots, \beta^+_k) (\beta^-_1, \dots, \beta^-_k)}$ and  $D^{(\lambda', \mu')}_{(\beta^+_1, \dots, \beta^+_k) (\beta^-_1, \dots, \beta^-_k)}$ are defined only for $k > 2$. For convenience we extend these definitions to the case $k=2$ by setting
\begin{align*}
&C^{(\lambda, \mu)}_{(\beta^+_1,\beta^+_2) (\beta^-_1,\beta^-_2)} = c^{(\lambda, \mu)}_{(\beta^+_1, \beta^+_2), (\beta^-_1, \beta^-_2)},\\
&D^{(\lambda', \mu')}_{(\beta^+_1,\beta^+_2) (\beta^-_1,\beta^-_2)} = d^{(\lambda', \mu')}_{(\beta^+_1, \beta^+_2), (\beta^-_1, \beta^-_2)}.
\end{align*}

Now we can combine (\ref{eqBranch1}) and (\ref{eqBranch2}) in the following proposition.
\begin{prop} \label{propBranRuleDiag} Let $\gl(n) \subset \gl(kn)$ be an embedding of signature $(k,0,0)$. Then
\begin{align*}
{V^{kn}_{\lambda, \mu}}_{\downarrow \gl(n)} \cong  \bigoplus_{\substack{\beta^+_1, \dots, \beta^+_k \\ \beta^-_1, \dots, \beta^-_k \\ \lambda', \mu'}}
C^{(\lambda, \mu)}_{(\beta^+_1, \dots, \beta^+_k) (\beta^-_1, \dots, \beta^-_k)} D^{(\lambda', \mu')}_{(\beta^+_1, \dots, \beta^+_k) (\beta^-_1, \dots, \beta^-_k)} V^n_{\lambda', \mu'}. 
\end{align*}
\end{prop}

\section{Embeddings of general tensor type}

The main goal of this section is to prove Theorem \ref{thmMainCC} below. 

We start by introducing some useful notations. Let $\g', \g$ be any two classical locally finite Lie algebras and let $\phi : \g' \rightarrow \g$ be an embedding of general tensor type. 
We define bases $\{v'_i\}_{i \in I}$ and $\{v'^*_i\}_{i \in I}$ for $V'$ and $V'_*$ in the following way. When $\g' \cong \gl(\infty), \sl(\infty)$, we take $I = \ZZ_{>0}$ and $\{v'_i\}_{i \in I}$ and $\{v'^*_i\}_{i \in I}$ to be an arbitrary pair of dual bases. When $\g' \cong \sp(\infty)$, we take $I = \ZZ \setminus \{0\}$ and $\{v'_i\}_{i \in I}$ to be a symplectic basis for $V'$, i.e. such that $\left\langle v'_i, v'_j \right\rangle = \sign(i) \delta_{i+j, 0}$. Then for the dual basis $\{v'^*_i\}_{i \in I}$ we obtain $v'^*_i = v'_{-i}$ and $v'^*_{-i} = -v'_i$ for any $i >0$. Similarly, when $\g' \cong \so(\infty)$, we take $I = \ZZ \setminus \{0\}$ and $\{v'_i\}_{i \in I}$ to be a basis of $V'$ with the property that $\left\langle v'_i, v'_j \right\rangle = \delta_{i+j, 0}$. Then for the dual basis $\{v'^*_i\}_{i \in I}$ we obtain $v'^*_i = v'_{-i}$ and $v'^*_{-i} = v'_i$ for any $i >0$.
 
Having fixed bases $\{v'_i\}_{i \in I}$ and $\{v'^*_i\}_{i \in I}$ as above, for any $i, j \in \ZZ_{>0}$ we define
\begin{align*} 
v'_i \tilde{\otimes} v'^*_j = \left\{ \begin{array}{ll}
         v'_i \otimes v'^*_j & \text{ if } \g' \cong  \gl(\infty), \sl(\infty) \\
        v'_i \otimes v'_{-j} + v'_{-j} \otimes v'_i &\text{ if } \g' \cong \sp(\infty) \\
				v'_i \otimes v'_{-j} - v'_{-j} \otimes v'_i &\text{ if } \g' \cong \so(\infty).\end{array} \right.
\end{align*}  

Then, $(v'_i \tilde{\otimes} v'^*_j) \cdot v'_s = \delta_{sj} v'_i$ and $(v'_i \tilde{\otimes} v'^*_j) \cdot v'^*_s = -\delta_{si} v'^*_j$.

Next, let 
\begin{equation} \label{eqDecSoc}
\begin{aligned}
&\soc_{\g'} V = \tilde{V}_1 \oplus \dots \oplus \tilde{V}_k \oplus \tilde{V}_{k+1} \oplus \dots \oplus \tilde{V}_{k+l} \oplus N_a, \\
&\soc_{\g'}V_* = \tilde{V}^*_{1} \oplus  \dots \oplus \tilde{V}^*_k \oplus \tilde{V}^*_{k+1} \oplus \dots \oplus \tilde{V}^*_{k+l} \oplus N_c,
\end{aligned}
\end{equation}
where $\tilde{V}_j \cong V'$ for $j = 1, \dots, k$ and $\tilde{V}_{k+j} \cong V'_*$ for $j = 1, \dots, l$. Similalry, $\tilde{V}^*_j \cong V'_*$ for $j = 1, \dots, k$ and $\tilde{V}^*_{k+j} \cong V'$ for $j = 1, \dots, l$. Let $\{v'_i\}_{i \in I}$, $\{v'^*_i\}_{i \in I}$ be as above. 
We construct the following bases of $V$ and $V_*$ respectively:
\begin{equation*} 
\begin{aligned}
&\{ \{v_i^j\}_{i \in I, j =1, \dots, k} \cup \{w_i^j\}_{i \in I, j =1, \dots, l} \cup \{z_i\}_{i \in I_a} \cup \{x_i\}_{i \in I_b} \}, \\
&\{ \{v_i^{j*}\}_{i \in I, j =1, \dots, k} \cup \{w_i^{j*}\}_{i \in I, j =1, \dots, l} \cup \{t_i\}_{i \in I_c} \cup \{y_i\}_{i \in I_d} \},
\end{aligned}
\end{equation*}
where $v^j_i$ denotes the image of $v'_i$ into $\tilde{V}_j$ for $j = 1, \dots, k$ and $w_i^j$ is the image of $v'^*_i$ into $\tilde{V}_{k+j}$ for $j = 1, \dots, l$. We proceed similarly with $v_i^{j*}$ and $w_i^{j*}$. Furthermore, $A = \{z_i\}_{i \in I_a}$  and $C = \{t_i\}_{i \in I_c}$ are bases respectively for $N_a$ and $N_c$ as submodules respectively of $V$ and $V_*$, where $I_a$  and $I_c$ are index sets with cardinalities $a = \dim N_a$ and $c = \dim N_c$. Similarly, $B = \{x_i\}_{i \in I_b}$ and $D = \{y_i\}_{i \in I_d}$ are bases of $N_b$ and $N_d$ considered as vector subspaces of $V$ and $V_*$, where $I_b$ and $I_d$ are index sets with cardinalities $b = \dim N_b$ and $d = \dim N_d$. 

Then we have the following proposition.

\begin{prop} \label{propBasis} Let $\g'$, $\g$, and $\phi$ be as above. We take decompositions of $\soc_{\g'}V$ and $\soc_{\g'} V_*$ as in (\ref{eqDecSoc}). Then $\tilde{V}_1 \oplus \dots \oplus \tilde{V}_{k+l}$ and $\tilde{V}^*_{1} \oplus  \dots \oplus \tilde{V}^*_{k+l}$ pair non-degenerately.
Furthermore, for each $i$ in the respective index set, $z_i$ pairs trivially with all elements from $\tilde{V}^*_{1} \oplus  \dots \oplus \tilde{V}^*_{k+l}$ and $x_i$ pairs non-degenerately with infinitely many elements from $\tilde{V}^*_{1} \oplus  \dots \oplus \tilde{V}^*_{k+l}$. Similarly, $t_i$ pairs trivially with all elements from  $\tilde{V}_1 \oplus \dots \oplus \tilde{V}_{k+l}$ and $y_i$ pairs non-degenerately with infinitely many elements from  $\tilde{V}_1 \oplus \dots \oplus \tilde{V}_{k+l}$.
\end{prop}

\begin{proof}
We fix decompositions of $\soc_{\g'}V$ and $\soc_{\g'}V_*$ as in (\ref{eqDecSoc}). Then for $i,j,m,n$ in the respective ranges we have
\begin{align*}
\left\langle v_i^j, t_m \right\rangle = \left\langle \phi(v'_i \tilde{\otimes} v'^*_n)\cdot v^j_n, t_m \right\rangle =  -\left\langle v^j_n, \phi(v'_i \tilde{\otimes} v'^*_n)\cdot t_m \right\rangle = 0.
\end{align*}

Suppose now that $v_i^j$ pairs trivially with all elements from $\soc_{\g'}V_*$. Then there exists $y_m$ such that $\left\langle v_i^j, y_m \right\rangle \neq 0$. But then
\begin{align*}
\left\langle v_i^j, y_m \right\rangle = \left\langle \phi(v'_i \tilde{\otimes} v'^*_i)\cdot v^j_i, y_m \right\rangle =  -\left\langle v^j_i, \phi(v'_i \tilde{\otimes} v'^*_i)\cdot y_m \right\rangle =
\left\langle v^j_i, v \right\rangle
\end{align*}
for some $v \in \soc_{\g'}V_*$. Hence, $\left\langle v^j_i, v \right\rangle \neq 0$, which contradicts our assumption. This implies that  $\tilde{V}_1 \oplus \dots \oplus \tilde{V}_{k+l}$ and $\tilde{V}^*_{1} \oplus  \dots \oplus \tilde{V}^*_{k+l}$ pair non-degenerately.

Suppose next that $y_m \in N_d$ pairs non-degenerately with finitely many elements from $\tilde{V}_1 \oplus \dots \oplus \tilde{V}_{k+l}$. Then, there exists $v \in \tilde{V}^*_1 \oplus \dots \oplus \tilde{V}^*_{k+l}$ such that $\tilde{y}_m = y_m - v$ pairs trivially with $\tilde{V}_1 \oplus \dots \oplus \tilde{V}_{k+l}$. But then $\phi(g') \cdot \tilde{y}_m = 0$ for all $g' \in \g'$ which contradicts with the definition of $N_d$. This proves the statement.
\end{proof}

We investigate further the properties of the decompositions of $\soc_{\g'}V$ and $\soc_{\g'} V_*$. Notice that
\begin{align*}
\left\langle v_i^j, v^{*n}_m \right\rangle = \left\langle \phi(v'_i \tilde{\otimes} v'^*_s)\cdot v^j_s, v^{*n}_m \right\rangle =  -\left\langle v^j_s, \phi(v'_i \tilde{\otimes} v'^*_s)\cdot v^{*n}_m \right\rangle = \delta_{im} \left\langle v^j_s, v^{*n}_s \right\rangle
\end{align*}
for all $i,j,m,n, s$. Hence, $\left\langle v_i^j, v^{*n}_m \right\rangle = 0$ for all $j,n$ and for $i \neq m$. In addition, 
\begin{align*} 
\left\langle v_i^j, v^{*n}_i  \right\rangle = \left\langle v^j_s, v^{*n}_s \right\rangle
\end{align*}
for all $i,j,n,s$.
In the same way, $\left\langle w_i^j, w^{*n}_m \right\rangle = 0$ for all $j,n$ and for $i \neq m$ and $\left\langle w_i^j, w^{*n}_i  \right\rangle = \left\langle v^j_s, v^{*n}_s \right\rangle$ for all $i,j, n, s$.

In a similar way we prove that  $\left\langle v_i^j, w^{*n}_m \right\rangle = 0$ and $\left\langle w_i^j, v^{*n}_m \right\rangle = 0$ for all $i,j,m,n$.

The above consderations and Proposition \ref{propBasis} imply that for any $j$ and $m$, the spaces $\tilde{V}_j$ and $\tilde{V}^*_m$ pair either trivially or non-degenerately. Furthermore, for each $j$ there exists $m$ such that $\tilde{V}_j$ and $\tilde{V}^*_m$ pair non-degenerately and $\{v_i^j\}$, $\{v^{*m}_i\}$ (resp., $\{w_i^j\}$, $\{w^{*m}_i\}$) is a pair of dual bases. 

The following proposition now holds.

\begin{prop} \label{propSLExclude} Let $\g$ be any classical locally finite Lie algebra and $M$ be a tensor $\g$-module.
\begin{itemize}
\item[(i)] Let $\g = \sl(V,V_*)$. Then any embedding $\sl(V', V'_*) \subset \g$ can be extended to an embedding $\gl(V', V'_*) \subset \gl(V, V_*)$ of general tensor type. Furthermore, the socle filtration of $M$ over $\sl(V', V'_*)$ is the same as the socle filtration of $M$ considered as a $\gl(V,V_*)$-module over $\gl(V',V'_*)$.
\item[(ii)] Let $\g \not\cong \sl(\infty)$. Then any embedding $\sl(V',V'_*) \subset \g$ can be extended to an embedding $\gl(V', V'_*) \subset \g$ of general tensor type. Furthermore,
$\soc_{\sl(V',V'_*)} ^{(r)} M = \soc_{\gl(V',V'_*)} ^{(r)} M$.
\end{itemize}
\end{prop}

\begin{proof}
Part (i): We fix decompositions of the $\sl(V, V_*)$-modules $V$ and $V_*$ such that for each $j$ there exists a unique $m$ with the property that $\tilde{V}_j$ and $\tilde{V}^*_m$ pair non-degenerately. These exist due to the discussion above. Let $\phi$ denote the embedding $\sl(V', V'_*)\subset \g$. Then
\begin{align*}
\phi(v'_i\otimes v'^*_s) = v^1_i\otimes v^{1*}_s + \dots + v^k_i\otimes v^{k*}_s - w^1_s\otimes w^{1*}_i - \dots - w^l_s\otimes w^{l*}_i,
\end{align*}
\begin{align*}
\phi(v_i' \otimes v'^*_i - v_s' \otimes v'^*_s) = &v^1_i \otimes v^{1*}_i - v^1_s \otimes v^{1*}_s + \dots + v^k_i \otimes v^{k*}_i - v^k_s \otimes v^{k*}_s + \\
&w^1_s\otimes w^{1*}_s - w^1_i\otimes w^{1*}_i + \dots + w^l_s \otimes w^{l*}_s - w^l_i\otimes w^{l*}_i
\end{align*}
for all $i \neq s$. We can naturally extend $\phi$ to an embedding $\phi: \gl(V', V'_*) \rightarrow \gl(V, V_*)$ by setting
\begin{align*}
\phi(v'_1\otimes v'^*_1) = v^1_1\otimes v^{1*}_1 + \dots + v^k_1\otimes v^{k*}_1 - w^1_1\otimes w^{1*}_1 - \dots - w^l_1\otimes w^{l*}_1.
\end{align*}
 
Next, let $M$ be as above. From \cite{PSt} we know that the set of tensor $\gl(V, V_*)$-modules coincides with the set of tensor $\sl(V,V_*)$-modules. Then $M$ is both a tensor $\gl(V,V_*)$-module and tensor $\sl(V, V_*)$-module. We will use the notation $M_{\gl}$ (resp., $M_{\sl}$) to mark that we consider $M$ as a $\gl(V, V_*)$ (resp., $\sl(V, V_*)$) module. Then, on the one hand, we have the chain of embeddings
\begin{align*}
\sl(V',V'_*) \subset \sl(V,V_*) \subset \gl(V, V_*).
\end{align*}
This chain yields the following equality for every $r$:
\begin{align} \label{eqEq1}
\soc_{\sl(V',V'_*)}^{(r)} M_{\sl} = \soc_{\sl(V',V'_*)}^{(r)} M_{\gl}.
\end{align}

On the other hand, we have the chain of embeddings
\begin{align*}
\sl(V',V'_*) \subset \gl(V',V'_*) \subset \gl(V, V_*),
\end{align*}
which yields the equality
\begin{align} \label{eqEq2}
\soc_{\sl(V',V'_*)}^{(r)} M_{\gl} = \soc_{\gl(V',V'_*)}^{(r)} M_{\gl}.
\end{align}
From (\ref{eqEq1}) and (\ref{eqEq2}) the statement follows.

The proof of part (ii) is analogous to the proof of part (i). 
\end{proof}

As a result of Proposition \ref{propSLExclude} we can exclude from our considerations below all cases of embeddings which involve $\sl(\infty)$, since they are equivalent to the respective cases which involve $\gl(\infty)$.

\begin {prop} \label{propDecompose}
Let $\g' \subset \g$ be an embedding of general tensor type. Then there exist intermediate subalgebras $\g_1$ and $\g_2$ of the same type as $\g$
such that $\g' \subset \g_2 \subset \g_1 \subset \g $ and the following hold.
\begin{itemize}
\item[(1)] The embedding $\g_1 \subset \g$ has the properties:
\begin{align*}
&\soc_{\g_1} V \cong V_1 \oplus N_{a_1},  \quad &V/\soc_{\g_1} V \cong N_b, \\
&\soc_{\g_1} V_* \cong V_{1*} \oplus N_{c_1},  \quad &V_*/\soc_{\g_1} V_* \cong N_d,
\end{align*}
where 
\begin{align*}
&N_{a_1} = \{v \in N_a | \left\langle v, N_c \right\rangle = 0 \}  \text{ and }
N_{c_1} = \{w \in N_c | \left\langle N_a, w \right\rangle = 0 \}.
\end{align*}
\item[(2)] The embedding $\g_2 \subset \g_1$ has the properties:
\begin{align*}
&V_1 \cong V_2 \oplus N_{a_2},  
&V_{1*} \cong V_{2*} \oplus N_{c_2}, 
\end{align*}
where $N_{a_2}$ and $N_{c_2}$ are such that $N_a = N_{a_1} \oplus N_{a_2}$ and $N_c = N_{c_1} \oplus N_{c_2}$.
\item[(3)] The embedding $\g' \subset \g_2$ has the properties:
\begin{align*}
&V_2 \cong kV' \oplus lV'_*,
&V_{2*} \cong lV' \oplus kV'_*. 
\end{align*}
\end{itemize}
\end{prop}

\begin{proof}
We take decompositions of $\soc_{\g'}V$ and $\soc_{\g'}V_*$ as in Proposition \ref{propBasis}. Set
\begin{align*}
&V_2 =  \tilde{V}_1 \oplus \dots \oplus \tilde{V}_{k+l}, \quad 
&V_{2*} = \tilde{V}^*_{1} \oplus  \dots \oplus \tilde{V}^*_{k+l}.
\end{align*}
Then Proposition \ref{propBasis} yields that $V_2$ and $V_{2*}$ pair non-degenerately. If $\g \cong \gl(\infty)$ we put $\g_2 = V_2 \otimes V_{2*}$. Similarly, 
if $\g \cong \sp(\infty)$ then $\g_2 = S^2V$, and if $\g \cong \so(\infty)$ then $\g_2  = \bigwedge^2 V$. 

Next, let $A$ and $C$ be as above. 
Let $A_1 = \{z'_i\}_{i \in I_{a_1}}$ consist of those elements in $A$ which pair trivially with all elements in $C$, and analogously let $C_1 = \{t'_i\}_{i \in I_{c_1}}$ consist of the elements in $C$ which pair trivially with all vectors in $A$. Let $A_2 = A \setminus A_1$ and $C_2 = C \setminus C_1$ and denote their elements respectively with $z''_i$ and $t''_i$ and their index sets with $I_{a_2}$ and $I_{c_2}$. 
Now set
\begin{align*}
&V_1 = V_2 \oplus \mathrm{span} \{z''_i\}_{i\in I_{a_2}}, 
&V_{1*} = V_{2*} \oplus \mathrm{span} \{t''_i\}_{i \in I_{c_2}}. 
\end{align*}
Then $V_1$ and $V_{1*}$ pair non-degenerately and if $\g \cong \gl(\infty)$ we put $\g_1 = V_1 \otimes V_{1*}$. We proceed analogously in the other cases.

The Lie algebras $\g_1$ and $\g_2$ satisfy the required properties. 
\end{proof}

We now state the main theorem for embeddings of general tensor type.

\begin{thm} \label{thmMainCC} Let $\g'\subset \g$ be an embedding of general tensor type. Let $M = V^{\{p,q\}}$ for $\g \cong \gl(\infty)$, $M = V^{\left\langle p \right\rangle}$ for $\g \cong \sp(\infty)$, and $M = V^{[p]}$ for $\g \cong \so(\infty)$. Let $\g_1$ and $\g_2$ be intermediate subalgebras of the same type as $\g$, such that the chain $\g'\subset \g_2 \subset \g_1 \subset \g$ satisfies the conditions of Proposition \ref{propDecompose}. Then, for any $N \subseteq M$, 
\begin{align*}
\overline{\soc}_{\g'}^{(r+1)} N \cong \bigoplus_{l+m+n=r} \overline{\soc}_{\g'}^{(l+1)} (\overline{\soc}_{\g_2}^{(m+1)}(\overline{\soc}_{\g_1}^{(n+1)} N)).
\end{align*}
\end{thm}

Motivated by Proposition \ref{propDecompose} and Theorem \ref{thmMainCC} we give the following definition.

\begin{df} \label{defEmbs} Let $\g',\g$ be a pair of classical locally finite Lie algebras.
\begin{itemize}
\item[(i)] An embedding $\g' \subset \g$ is said to be of type I if $\g' \cong \g$ and 
\begin{align*}
&\soc_{\g'} V \cong V' \oplus N_{a},  \quad &V/\soc_{\g'} V \cong N_b, \\
&\soc_{\g'} V_* \cong V'_{*} \oplus N_{c},  \quad &V_*/\soc_{\g'} V_* \cong N_d,
\end{align*}
where $N_{a}$ and $N_{c}$ pair trivially.
\item[(ii)] An embedding $\g' \subset \g$ is said to be of type II if $\g' \cong \g$ and 
\begin{align*}
&V \cong V' \oplus N_{a},  
&V_{*} \cong V'_{*} \oplus N_{c}.
\end{align*}
\item[(iii)] An embedding $\g' \subset \g$ is said to be of type III if 
\begin{align*}
&V \cong kV' \oplus lV'_*,
&V_{*} \cong lV' \oplus kV'_*. 
\end{align*}
\end{itemize}
\end{df}

Before proving Theorem \ref{thmMainCC} we need some preparational work. More precisely, in the following paragraphs we describe the socle filtrations over $\g'$ of the $\g$-modules $V^{\{p,q\}}$, $V^{\left\langle p \right\rangle}$, and $V^{[p]}$ when we have embeddings of type I, II, III. When $\g' \subset \g$ is of type I or II, to simplify notations we will consider that $V' \subset V$ and $V'_* \subset V_*$ and we will not make use of the notations $\tilde{V}_j, \tilde{V}^*_j$, introduced in (\ref{eqDecSoc}).

Let first $\g' \subset \g$ be an embedding of type I. We have the following short exact sequence of $\g'$-modules
\begin{align} \label{eqSES1}
0 \rightarrow V'\oplus N_a \stackrel{i}{\rightarrow} V \stackrel{f}{\rightarrow} N_b \rightarrow 0.
\end{align}
For each index $1 \leq i \leq p$ we define
\begin{align*}
&L_i : V^{\otimes p} \rightarrow V^{\otimes i-1}\otimes N_b \otimes V^{\otimes p-i} 
,\\
&L_i = \id \otimes \id \otimes \dots \otimes \id \otimes f \otimes \id \otimes \dots \otimes \id,
\end{align*}
where $f$ appears at the $i$-th position in the tensor product. In what follows, without further reference we will use the standard isomorphism $V^{\otimes i-1}\otimes N_b \otimes V^{\otimes p-i} \cong V^{\otimes p-1} \otimes N_b$ and consider $L_i : V^{\otimes p} \rightarrow V^{\otimes p-1} \otimes N_b$.

Similarly, for each collection of indices $1 \leq i_1 < i_2 < \dots < i_k \leq p$ we define
\begin{align*}
&L_{i_1, \dots, i_k} : V^{\otimes p} \rightarrow 
V^{\otimes p-k}\otimes N_b^{\otimes k}, \\
&L_{i_1, \dots, i_k} = \id \otimes \dots \otimes \id \otimes f \otimes \id \otimes \dots \otimes \id \otimes f \otimes \id \dots \otimes \id,
\end{align*}
where the map $f$ appears at positions $i_1$ through $i_k$ in the tensor product. When $k = 0$, we set $L_0 = \id$.

In a similar way, when $\g',\g \cong \gl(\infty)$ we take also the short exact sequence
$$
0 \rightarrow V_*'\oplus N_c \stackrel{i}{\rightarrow} V_* \stackrel{g}{\rightarrow} N_d \rightarrow 0
$$
and define
\begin{align*}
&M_{i_1, \dots, i_k} : V_*^{\otimes q} \rightarrow 
V_*^{\otimes q-k}\otimes N_d^{\otimes k}, \\
&M_{i_1, \dots, i_k} = \id \otimes \dots \otimes \id \otimes g \otimes \id \otimes \dots \otimes \id \otimes g \otimes \id \otimes \dots \otimes \id,
\end{align*}
where the map $g$ appears at positions $i_1$ through $i_k$ in the tensor product. When $k = 0$, we set $M_0 = \id$.

Next, let $M = V^{\{p,q\}}$ (respectively, $V^{\left\langle p \right\rangle}$, $V^{[p]}$). Then we have the following proposition.

\begin{prop} \label{propEmb1}
Let $\g' \subset \g$ be an embedding of type I. Then
\begin{itemize}
\item [(i)] if $\g \cong \gl(\infty)$
\begin{align*}
\soc_{\g'}^{(r)} M = \bigcap_{\substack{n_1 +n_2 = r \\ i_1 < \dots < i_{n_1} \\ j_1 < \dots <j_{n_2} }} \ker ({L_{i_1, \dots, i_{n_1}} \otimes M_{j_1, \dots, j_{n_2}} }_{\vert M}) =  \bigcap_{\substack{n_1 +n_2 = r \\ i_1 < \dots < i_{n_1} \\ j_1 < \dots <j_{n_2} }} (\ker {L_{i_1, \dots, i_{n_1}} \otimes M_{j_1, \dots, j_{n_2}} }) \cap M;
\end{align*}
\item [(ii)] if $\g \cong \sp(\infty), \so(\infty)$
\begin{align*}
\soc_{\g'}^{(r)} M = \bigcap_{i_1 < \dots < i_{r}} \ker ({L_{i_1, \dots, i_{r}}}_{\vert M}) = \bigcap_{i_1 < \dots < i_{r}} (\ker {L_{i_1, \dots, i_{r}}}) \cap M.
\end{align*}
\end{itemize}
\end{prop}

\begin{proof}
Part (i): Denote 
$$S^{(r)}_{p,q} = \bigcap_{\substack{n_1 +n_2 = r \\ i_1 < \dots < i_{n_1} \\ j_1 < \dots <j_{n_2} }} \ker {L_{i_1, \dots, i_{n_1}} \otimes M_{j_1, \dots, j_{n_2}} }.
$$
Then $S_{p,q}^{(r+1)}$ consists of all elements $u \in V^{\otimes (p,q)}$ such that each monomial in the expression of $u$ contains at most $n_1$ entries outside of $\soc_{\g'} V$ and at most $n_2$ entries outside of $\soc_{\g'}V_*$ for some $n_1 + n_2 = r$. In particular, $S_{p,q}^{(1)} = (\soc_{\g'}V)^{\otimes p} \otimes (\soc_{\g'}V_*)^{\otimes q}$. To show that $S^{(r)}_{p,q} = \soc_{\g'}^{(r)} V^{\{p,q\}}$ we need to check two properties:
\begin{itemize}
\item[(a)] for any $ u \in (S_{p,q}^{(r+2)} \cap V^{\{p,q\}}) \setminus (S_{p,q}^{(r+1)} \cap V^{\{p,q\}})$ there exists $g \in U(\g')$ (where $U(\g')$ is the universal enveloping algebra of $\g'$) such that $g \cdot u \in (S_{p,q}^{(r+1)} \cap V^{\{p,q\}}) \setminus (S_{p,q}^{(r)} \cap V^{\{p,q\}})$;
\item[(b)] the quotient $(S_{p,q}^{(r+1)} \cap V^{\{p,q\}}) / (S_{p,q}^{(r)} \cap V^{\{p,q\}})$ is semisimple for any $r$.
\end{itemize}

Proof of (b): Notice that
\begin{align*}
S_{p,q}^{(1)} \cap V^{\{p,q\}} = ((V' \oplus N_a)^{\otimes p} \otimes (V'_* \oplus N_c)^{\otimes q}) \cap V^{\{p,q\}}.
\end{align*}
Since $N_a$ and $N_c$ pair trivially, this implies that
\begin{align} \label{eqS1} 
S_{p,q}^{(1)} \cap V^{\{p,q\}} \cong \bigoplus_{k=0}^p \bigoplus_{l = 0}^q \binom{p}{k} \binom {q}{l} N_a^{\otimes k} \otimes N_c^{\otimes l} \otimes V'^{\{p-k,q-l\}},
\end{align}
which is a semisimple $\g'$-module. 

Next, for any fixed $n_1$ and $n_2$ and sequences $i_1 < \dots <i_{n_1}$ and $j_1 < \dots < j_{n_2}$
\begin{align*}
L_{i_1, \dots, i_{n_1}} \otimes M_{j_1, \dots, j_{n_2}} : V^{\otimes (p,q)} \rightarrow N_b^{\otimes n_1} \otimes N_d^{\otimes n_2} \otimes V^{\otimes (p-n_1, q-n_2)}.
\end{align*} 
Moreover, if $n_1 + n_2 = r$ for some $r$ then 
\begin{align*}
L_{i_1, \dots, i_{n_1}} \otimes M_{j_1, \dots, j_{n_2}} (S_{p,q}^{(r+1)}) \subseteq N_b^{\otimes n_1} \otimes N_d^{\otimes n_2} \otimes S_{p-n_1,q-n_2}^{(1)}.
\end{align*}
Hence,
\begin{align*}
L_{i_1, \dots, i_{n_1}} \otimes M_{j_1, \dots, j_{n_2}} (S_{p,q}^{(r+1)} \cap V^{\{p,q\}}) \subseteq N_b^{\otimes n_1} \otimes N_d^{\otimes n_2} \otimes (S_{p-n_1,q-n_2}^{(1)} \cap V^{\{p-n_1,q-n_2\}}).
\end{align*}

Thus for any $r$ we have a well-defined injective homomorphism of $\g'$-modules
\begin{align*}
\bigoplus_{\substack{n_1 + n_2 = r \\ i_1 < \dots < i_{n_1} \\ j_1 < \dots < j_{n_2}}} L_{i_1, \dots, i_{n_1}} \otimes M_{j_1, \dots, j_{n_2}} : S_{p,q}^{(r+1)} \cap V^{\{p,q\}}/S_{p,q}^{(r)} \cap V^{\{p,q\}}  \rightarrow \\
\bigoplus_{\substack{n_1 + n_2 = r \\ i_1 < \dots < i_{n_1} \\ j_1 < \dots < j_{n_2}}} N_b^{\otimes n_1} \otimes N_d^{\otimes n_2} \otimes (S_{p-n_1,q-n_2}^{(1)} \cap V^{\{p-n_1,q-n_2\}}).
\end{align*}
This proves (b).

Proof of (a): As before we take bases of $V$ and $V_*$ of the form
\begin{align*}
& \{\{v'_i\}_{i \in I} \cup \{z_i\}_{i \in I_a} \cup \{x_i\}_{i \in I_b}\}, \quad & \{\{v'^*_i\}_{i \in I} \cup \{t_i\}_{i \in I_c} \cup \{y_i\}_{i \in I_d}\}.
\end{align*}
We take also $u \in S^{(r+2)}_{p,q} \cap V^{\{p,q\}}$ and we write it as $u = \sum_{i=1}^N a_{i} u_i \otimes u^*_i$, where each $u_i$ is a monomial in $V^{\otimes p}$ and each $u^*_i$ is a monomial in $V_*^{\otimes q}$. Then we take $g_1 = v'_{k_1} \otimes v'^*_{l_1}$ such that $v'_{k_1}$ satisfies: $\left\langle v'_{k_1}, y_{j}\right\rangle \neq 0$ for at least one $y_{j}$ that enters the monomial $u^*_1$; $ \left\langle v'_{k_1}, v'^*_{j}\right\rangle = 0$ for all $v'^*_{j}$ in $u^*_1$; and $v'_{k_1}$ does not appear in any $u_i$. Similarly, $v'^*_{l_1}$ satisfies:  $\left\langle x_{j}, v'^*_{l_1}\right\rangle \neq 0$ for at least one $x_{j}$ that enters the monomial $u_1$; $\left\langle v'_{j}, v'^*_{l_1}\right\rangle = 0$ for all $v'_{j}$ in $u_1$; and $v'^*_{l_1}$ does not appear in any $u_i^*$. Then 
$$
g_1 \cdot (u_1 \otimes u_1^*) \in (S_{p,q}^{(r+1)} \cap V^{\{p,q\}}) \setminus (S_{p,q}^{(r)} \cap V^{\{p,q\}}).
$$
After defining inductively $g_1, \dots g_{i-1}$, if $(g_{i-1} \circ \dots \circ g_{1}) \cdot u_i\otimes u_i^*$ is not in the desired space, we define $g_i$ in the same way as we defined $g_1$. Finally we set $g = g_N \circ \dots \circ g_1$ and then $g \cdot u$ satisfies (a). 

The proof of part (ii) is analogous (and actually simpler).
\end{proof}

{\bf Remark.} Notice that if $q=0$ then (\ref{eqS1}) can be rewritten as
\begin{align*}
S_{p,0}^{(1)} \cong \bigoplus_{k=0}^p \binom{p}{k} N_a^{\otimes k} \otimes V'^{\otimes (p-k)},
\end{align*}
which is semisimple without any conditions on the pairing between $N_a$ and $N_c$. Thus, for the $\gl(\infty)$-modules $V^{\otimes p}$ and $V^{\otimes q}$ Proposition \ref{propEmb1} holds even without the requirement that $N_a$ and $N_c$ pair trivially.

\begin{coro} \label{corSocEmb1} If $\g', \g \cong \gl(\infty)$ then 
\begin{align*}
\soc_{\g'}^{(r+1)} V^{\{p,q\}} =(\sum_{n_1 + n_2 = r} (\soc_{\g'}^{(n_1 + 1)} V^{\otimes p}) \otimes (\soc_{\g'}^{(n_2 + 1)} V_*^{\otimes q})) \cap V^{\{p,q\}}.
\end{align*}
\end{coro}
\begin{proof}
 Notice that
\begin{align*}
\bigcap_{\substack{n_1 +n_2 = r+1 \\ i_1 < \dots < i_{n_1} \\ j_1 < \dots <j_{n_2} }} \ker (L_{i_1, \dots, i_{n_1}} \otimes M_{j_1, \dots, j_{n_2}}) = \sum_{n_1 + n_2 = r}(\bigcap_{i_1 < \dots < i_{n_1 + 1}} \ker L_{i_1, \dots, i_{n_1 + 1}}) \otimes (\bigcap_{j_1 < \dots < j_{n_2 + 1}} M_{j_1, \dots, j_{n_2 +1}}).
\end{align*}
This proves the statement.
\end{proof}

Next, let $\g' \subset \g$ be an embedding of type II. Then $N_a$ and $N_c$ have the same dimension and there is a non-degenerate bilinear pairing between them, which is the restriction of the bilinear pairing between $V$ and $V_*$. Let $\{z_i\}_{i \in I_a}$ and $\{t_i\}_{i \in I_a}$ be a pair of dual bases for $N_a$ and $N_c$. We define a new bilinear form:
\begin{align*}
\left\langle \cdot, \cdot \right\rangle_t : V \times V_* \rightarrow \CC 
\end{align*}
such that $\left\langle z_i, t_j\right\rangle_t = \delta_{ij}$ and for all other pairs of basis elements from $V\times V_*$ the bilinear form is trivial. If $\g \cong \gl(\infty)$, for any pair of indices $I = (i,j)$ with $i \in \{1,2, \dots, p\}$ and $j \in \{1,2, \dots, q\}$, we define the contraction 
\begin{align*}
\Phi_I: V^{\otimes (p,q)} \rightarrow V^{\otimes (p-1, q-1)},
\end{align*}
\begin{align*}
&u_1\otimes \dots \otimes u_p \otimes u_1^* \otimes \dots \otimes u_q^* \mapsto
\left\langle u_i, u_j^* \right\rangle_t u_1\otimes \dots \otimes \hat{u_i}\otimes \dots \otimes u_p \otimes u_1^* \otimes \dots \otimes \hat{u^*_j}\otimes \dots \otimes u_q^*
\end{align*}
for any $u_i \in V$, $u^*_i \in V_*$. Similarly, for any collection of pairwise disjoint index pairs $I_1, \dots, I_r$ respectively from the sets $\{1,2, \dots, p\}$ and $\{1,2, \dots, q\}$, we define the $r$-fold contraction
\begin{align*}
\Phi_{I_1, \dots, I_r}: V^{\otimes (p,q)} \rightarrow V^{\otimes (p-r, q-r)}
\end{align*}
in the obvious way.

If $\g \cong \sp(\infty), \so(\infty)$, for any pair of indices $I = (i,j)$ with $i < j \in \{1,2, \dots, p\}$ we define the contraction 
\begin{align*}
\Phi_I: V^{\otimes p} \rightarrow V^{\otimes (p-2)},
\end{align*}
\begin{align*}
&u_1\otimes \dots \otimes u_p \mapsto \left\langle u_i, u_j \right\rangle_t u_1\otimes \dots \otimes \hat{u_i}\otimes \dots \otimes \hat{u_j} \otimes \dots \otimes u_p
\end{align*}
for $u_i \in V$. In a similar way we define the $r$-fold contraction $\Phi_{I_1, \dots, I_r}: V^{\otimes p} \rightarrow V^{\otimes (p-2r)}$.

Next, let $M = V^{\{p,q\}}$ (repsectively, $V^{\left\langle p \right\rangle}$, $V^{[p]}$) and let the restriction of $\Phi_{I_1, \dots, I_r}$ to $M$ be denoted by $\Phi_{I_1, \dots, I_r}$ again. Then we have the following proposition.

\begin{prop} \label{propEmb2Appr1}Let $\g' \subset \g$ be an embedding of type II. Then
\begin{align*}
\soc^{(r)}_{\g'} M = \bigcap_{I_1, \dots, I_r} \ker \Phi_{I_1, \dots, I_r}.
\end{align*}
\end{prop}

\begin{proof}
We prove it for the case $\g' \cong \g \cong \gl(\infty)$. For the other two cases the proof is analogous.
Denote similarly as before $ S_{p,q}^{(r)} = \bigcap_{I_1, \dots, I_r} \ker \Phi_{I_1, \dots, I_r}$. Then, 
\begin{align} \label{eqEmb2Soc}
&S_{p,q}^{(1)} \cong \bigoplus_{k=0}^p \bigoplus_{l=0}^q \binom{p}{k} \binom{q}{l} N_a^{\{k,l\}}\otimes V'^{\{p-k, q-l\}},
\end{align}
where for any $k,l$
\begin{align*}
N_a^{\{k,l\}} = (\bigcap_{I} \ker \Phi_{I}) \cap (N_a^{\otimes k} \otimes N_c^{\otimes l}).
\end{align*}
This proves that $S_{p,q}^{(1)}$ is a semisimple $\g'$-module.
Then for each disjoint collection of index pairs $I_1, \dots, I_r$
\begin{align*}
\Phi_{I_1, \dots, I_r} (S_{p,q}^{(r+1)}) \subseteq S_{p-r,q-r}^{(1)}.
\end{align*}
In this way the following is a well-defined and injective homomoprhism of $\g'$-modules
\begin{align}\label{eqEmb2SS}
\bigoplus_{\{I_1, \dots, I_r\}}\Phi_{I_1, \dots, I_r} : S_{p,q}^{(r+1)} / S_{p,q}^{(r)} \rightarrow \bigoplus_{\{I_1, \dots, I_r\}} S_{p-r,q-r}^{(1)}.
\end{align}
Here $\{I_1, \dots, I_r\}$ denotes the set of index pairs $I_1, \dots, I_r$ without considering their order. Then from (\ref{eqEmb2SS}) the semisimplicity of the consecutive quotients follows.

To show that this filtration is indeed the socle filtration of $V^{\{p,q\}}$, we take $u \in S_{p,q}^{(r+2)} \setminus S_{p,q}^{(r+1)}$. Then without loss of generality, $u = u_1 + \dots + u_s$ for some $s$ such that
\begin{align*}
u_1 = &z_{i_1}\otimes \dots \otimes z_{i_{r+1}}\otimes u_{11}\otimes t_{i_1} \otimes \dots \otimes t_{i_{r+1}} - \\
 &v'_k\otimes z_{i_2}\otimes \dots \otimes z_{i_{r+1}}\otimes u_{11}\otimes v'^*_k\otimes t_{i_2}\otimes \dots\otimes t_{i_{r+1}} - \\
&z_{i_1}\otimes v'_k\otimes z_{i_3} \otimes \dots \otimes z_{i_{r+1}}\otimes u_{11}\otimes t_{i_1}\otimes v'^*_k \otimes t_{i_3} \otimes \dots\otimes t_{i_{r+1}} - \dots - \\
&z_{i_1}\otimes \dots \otimes z_{i_{r}}\otimes v'_k\otimes u_{11}\otimes t_{i_1}\otimes \dots\otimes t_{i_{r}} \otimes v'^*_k + u_{12}
\end{align*}
where $v'_k$, $v'^*_k$ is as before a pair of dual basis elements from $V'\otimes V'_*$, $u_{11} \in V'^{\{p-r-1, q-r-1\}}$, and $u_{12} \in S_{p,q}^{(r+1)}$. The elements $u_2, \dots, u_s$ have a similar form. Notice that if $v'_k$ appears at most $m$ times in any monomial in the expression of $u_{11}$, then it appears at most $m+1$ times in any monomial in $u_1$. Hence, for $j = 1, \dots, m+1$, if we take
\begin{align*}
g_j = v'_{i_j} \otimes v'^*_k
\end{align*}
such that for all $j$, $v'_{i_j} \in V'$ does not appear in the expression of $u$ and pairs trivially with all entries in $u_{11}$, then 
$(g_1 \circ \dots \circ g_{m+1}) (u_1) \in S_{p,q}^{(r+1)} \setminus S_{p,q}^{(r)}$. We can do the same procedure for $u_2, \dots, u_s$ and this completes the proof.
\end{proof}

Next, let $\g' \subset \g$ be an embedding of type III. Let
\begin{align*}
&V = \tilde{V}_1 \oplus \dots \oplus \tilde{V}_k \oplus \tilde{V}_{k+1} \oplus \dots \oplus \tilde{V}_{k+l}, \\
&V_* = \tilde{V}^*_1 \oplus \dots \oplus \tilde{V}^*_k \oplus \tilde{V}^*_{k+1} \oplus \dots \oplus \tilde{V}^*_{k+l}
\end{align*}
be decompositions of $V$ and $V_*$ as in Proposition \ref{propBasis} with their respective bases. Then
\begin{align*}
V^{\otimes (p,q)} = \bigoplus_{\substack{i_1, \dots, i_p \\ j_1, \dots, j_q}} \tilde{V}_{i_1}\otimes \dots \otimes \tilde{V}_{i_p}\otimes \tilde{V}^*_{j_1}\otimes \dots \otimes \tilde{V}^*_{j_q},
\end{align*}
where $i_1, \dots, i_p, j_1, \dots, j_q = 1, \dots, k+l$ are not necessarily distinct indices. Notice that each $\tilde{V}_{i_1}\otimes \dots \otimes \tilde{V}_{i_p}\otimes \tilde{V}^*_{j_1}\otimes \dots \otimes \tilde{V}^*_{j_q}$ is isomorphic to $V'^{\otimes (p',q')}$ for some $p'$ and $q'$ with $p'+q' = p+q$. 

Let $J_1, \dots, J_s$ be a collection of disjoint index pairs $(i,j)$, where $i \neq j = 1, \dots p+q$. We want to define a map $\tilde{\Phi}_{J_1, \dots, J_s}$ on $V^{\otimes (p,q)}$ such that on each subspace $\tilde{V}_{i_1}\otimes \dots \otimes \tilde{V}_{i_p}\otimes \tilde{V}^*_{j_1}\otimes \dots \otimes \tilde{V}^*_{j_q}$, $\tilde{\Phi}_{J_1, \dots, J_s}$ is the contraction with respect to the bilinear form on $V' \times V'_*$. 

We proceed in the following way. We define a bilinear form $\left\langle \cdot, \cdot\right\rangle_d : (V \oplus V_*)\times (V \oplus V_*) \rightarrow \CC$ 
by setting
\begin{equation} \label{eqFormD}
\begin{aligned}
&\left\langle v^{s_1}_i, v^{*s_2}_j \right\rangle_d = \delta_{ij}, \qquad
\left\langle w^{*s_3}_i, w^{s_4}_j \right\rangle_d = \delta_{ij}, \\
&\left\langle v^{s_1}_i, w^{s_3}_j \right\rangle_d = \delta_{ij}, \qquad
\left\langle w^{*s_3}_i, v^{*s_1}_j \right\rangle_d = \delta_{ij}
\end{aligned}
\end{equation}
for all $s_1, s_2 = 1, \dots, k$ and all $s_3, s_4 = 1, \dots, l$. 
In addition, all other pairings of basis elements are trivial.
Let $J_1, \dots, J_s$ be as above. Let
\begin{align*} 
\Phi'_{J_1, \dots, J_s} : V^{\otimes (p,q)} \rightarrow V^{\otimes (p-s',q-s'')}
\end{align*}
be the contraction with respect to the bilinear form $\left\langle \cdot, \cdot\right\rangle_d$, where $s' + s'' = 2s$. In the case $\g \cong \sp(\infty)$ or $\g \cong \so(\infty)$ the above can be rewritten as
\begin{align*} 
\Phi'_{J_1, \dots, J_s} : V^{\otimes p} \rightarrow V^{\otimes (p-2s)}.
\end{align*}
It is not difficult to prove that $\Phi'_{J_1, \dots, J_s}$ is a $\g'$-module homomorphism. Notice that $\Phi'_{J_1, \dots, J_s}$ acts on $\tilde{V}_{i_1}\otimes \dots \otimes \tilde{V}_{i_p}\otimes \tilde{V}^*_{j_1}\otimes \dots \otimes \tilde{V}^*_{j_q}$ either as the usual contraction with respect to the bilinear form on $V' \times V'_*$ or as the zero map.

Next, let $\pi_{\substack{i_1, \dots, i_p \\ j_1, \dots, j_q}}$ denote the projection of $V^{\otimes (p,q)}$ onto the subspace $\tilde{V}_{i_1}\otimes \dots \otimes \tilde{V}_{i_p}\otimes \tilde{V}^*_{j_1}\otimes \dots \otimes \tilde{V}^*_{j_q}$. Then we set
\begin{align*}
\tilde{\Phi}_{J_1, \dots, J_s} = \bigoplus_{\substack{i_1, \dots, i_p \\ j_1, \dots, j_q}} \Phi'_{J_1, \dots, J_s} \circ \pi_{\substack{i_1, \dots, i_p \\ j_1, \dots, j_q}}.
\end{align*}

Now the following proposition holds.

\begin{prop} \label{propGenEmb3} Let $\g' \subset \g$ be an embedding of type III.
Then
\begin{align*}
\soc_{\g'}^{(r)} V^{\otimes (p,q)} = \bigcap_{J_1, \dots, J_r} \ker \tilde{\Phi}_{J_1, \dots, J_r}.
\end{align*}
\end{prop}

\begin{proof}
Property (\ref{eqPropSoc}) of socle filtration implies 
\begin{align*}
\soc_{\g'}^{(r)} V^{\otimes (p,q)} = \bigoplus_{\substack{i_1, \dots, i_p \\ j_1, \dots, j_q}} \soc_{\g'}^{(r)} \tilde{V}_{i_1}\otimes \dots \otimes \tilde{V}_{i_p}\otimes \tilde{V}^*_{j_1}\otimes \dots \otimes \tilde{V}^*_{j_q}.
\end{align*}

Moreover, from Theorem 2.2 in \cite{PSt} it follows that
\begin{align*}
\soc_{\g'}^{(r)} \tilde{V}_{i_1}\otimes \dots \otimes \tilde{V}_{i_p}\otimes \tilde{V}^*_{j_1}\otimes \dots \otimes \tilde{V}^*_{j_q} = \bigcap_{J_1,\dots, J_r} \ker \Phi'_{J_1, \dots, J_r} | _{\tilde{V}_{i_1}\otimes \dots \otimes \tilde{V}_{i_p}\otimes \tilde{V}^*_{j_1}\otimes \dots \otimes \tilde{V}^*_{j_q}}.
\end{align*}

Then,
\begin{align*}
&\soc_{\g'}^{(r)} V^{\otimes (p,q)} = \bigoplus_{\substack{i_1, \dots, i_p \\ j_1, \dots, j_q}}\bigcap_{J_1,\dots, J_r} \ker \Phi'_{J_1, \dots, J_r} | _{\tilde{V}_{i_1}\otimes \dots \otimes \tilde{V}_{i_p}\otimes \tilde{V}^*_{j_1}\otimes \dots \otimes \tilde{V}^*_{j_q}} = \\
& \bigcap_{\substack{i_1, \dots, i_p \\ j_1, \dots, j_q}}\bigcap_{J_1,\dots, J_r} \ker (\Phi'_{J_1, \dots, J_r} \circ \pi_{\substack{i_1, \dots, i_p \\ j_1, \dots, j_q}})= \bigcap_{J_1,\dots, J_r} \ker \tilde{\Phi}_{J_1, \dots, J_r}.
\end{align*}
\end{proof}

We are now ready to prove Theorem \ref{thmMainCC}.

\begin{proof}[Proof of Theorem \ref{thmMainCC}]

We fix bases of $V$ and $V_*$ as in Proposition \ref{propBasis}. We take the short exact sequences:
\begin{align*}
&0 \rightarrow kV'\oplus lV'_* \oplus N_{a_1} \oplus N_{a_2} \stackrel{i}{\rightarrow} V \stackrel{f}{\rightarrow} N_b \rightarrow 0,\\
&0 \rightarrow lV'\oplus kV'_* \oplus N_{c_1} \oplus N_{c_2} \stackrel{i}{\rightarrow} V_* \stackrel{g}{\rightarrow} N_d \rightarrow 0
\end{align*}
and define the homomorphisms $L_{i_1,\dots, i_{n_1}}$ and $M_{j_1, \dots, j_{n_2}}$ as in the settings for embeddings of type I. 
For $\g \cong \sp(\infty), \so(\infty)$ we always take $n_2 = 0$. Recall that if $M = V^{\{p,q\}}$ then
\begin{align*}
L_{i_1,\dots, i_{n_1}}\otimes M_{j_1, \dots, j_{n_2}} : V^{\{p,q\}} \rightarrow N_b^{\otimes n_1} \otimes N_d^{\otimes n_2} \otimes V^{\{p-n_1, q-n_2\}}
\end{align*}
and similarly for the other cases.

Next, we define a pairing $\left\langle \cdot, \cdot \right\rangle_t$ on $V \times V_{*}$ as in the settings for embeddings of type II. More precisely, we set
\begin{align*}
\left\langle z_i, t_j \right\rangle_t = \left\langle z_i, t_j \right\rangle = \delta_{ij}
\end{align*}
if $z_i \in N_{a_2}$ and $t_j \in N_{c_2}$. Moreover, we take all other couples of basis elements from $V\times V_*$ to pair trivially. Then we define the contraction maps $
\Phi_{I_1, \dots, I_m} : M \rightarrow M'$ with respect to the pairing $\left\langle \cdot,\cdot \right\rangle_t$. Here $M' = V^{\{p-m, q-m\}}$ for $M = V^{\{p,q\}}$, and similarly for the other cases. If $m=0$ we set $\Phi_0 = \id$. It is not difficult to see that for any choice of $m$ and of disjoint index pairs $I_1, \dots, I_m$, the maps $\Phi_{I_1, \dots, I_m}$ are homomorphisms of $\g_2$-modules and hence also of $\g'$-modules.

Similarly, we define $\Phi'_{J_1, \dots, J_l} : M \rightarrow M''$ to be the contraction with respect to the bilinear form $\left\langle \cdot,\cdot \right\rangle_d$ defined by (\ref{eqFormD})
and such that all other basis elements from $(V\oplus V_*)\times (V\oplus V_*)$ 
pair trivially. As in the settings for Proposition \ref{propGenEmb3} we set 
$$
\pi_{\substack{i_1, \dots, i_p \\ j_1, \dots, j_q}} : M \rightarrow \tilde{V}_{i_1}\otimes \dots \otimes \tilde{V}_{i_p}\otimes \tilde{V}^*_{j_1}\otimes \dots \otimes \tilde{V}^*_{j_q}.
$$
Note that the map $\pi_{\substack{i_1, \dots, i_p \\ j_1, \dots, j_q}}$ is not a $\g'$-module homomorphism, it is just a linear map. Then as before we define
\begin{align*}
\tilde{\Phi}_{J_1, \dots, J_l} = \bigoplus_{\substack{i_1, \dots, i_p \\ j_1, \dots, j_q}} \Phi'_{J_1, \dots, J_l} \circ \pi_{\substack{i_1, \dots, i_p \\ j_1, \dots, j_q}},
\end{align*}
which again is just a linear map. If $l=0$ we set $\tilde{\Phi}_0 = \id$.

Note that we defined the maps $\tilde{\Phi}_{J_1, \dots, J_l}$ and $\Phi_{I_1, \dots, I_m}$ for $M$ being equal to $V^{\{p,q\}}$, $V^{\left\langle p \right\rangle}$, and $V^{[p]}$. To simplify notations, when we have modules of the form $N_b^{\otimes n_1}\otimes N_d^{\otimes n_2}\otimes M$ we will denote the maps $\id \otimes \id \otimes \tilde{\Phi}_{J_1, \dots, J_l}$ and $\id \otimes \id \otimes \Phi_{I_1, \dots, I_m}$ again respectively by $\tilde{\Phi}_{J_1, \dots, J_l}$ and $\Phi_{I_1, \dots, I_m}$.

Now, for any $r \geq 0$ we define
\begin{align*}
S^{(r)}(M) = \bigcap_{\substack{l+m+n = r \\ n_1 + n_2 = n \\ i_1,\dots, i_{n_1}, j_1, \dots, j_{n_2} \\ I_1, \dots, I_m \\ J_1,\dots, J_l} } \ker \tilde{\Phi}_{J_1, \dots, J_l} \circ \Phi_{I_1, \dots, I_m} \circ (L_{i_1, \dots, i_{n_1}} \otimes M_{j_1, \dots, j_{n_2}}).
\end{align*}
If $\g \cong \sp(\infty), \so(\infty)$ the above can be rewritten as
\begin{align*}
S^{(r)}(M) = \bigcap_{\substack{l+m+n = r\\ i_1,\dots, i_{n}\\ I_1, \dots, I_m \\ J_1,\dots, J_l} } \ker \tilde{\Phi}_{J_1, \dots, J_l} \circ \Phi_{I_1, \dots, I_m} \circ L_{i_1, \dots, i_{n}}.
\end{align*}
For shortness, we write $S^{(r)}$ instead of $S^{(r)}(M)$ when $M$ is clear from the context.

Now the following three properties hold:
\begin{itemize}
\item [(1)] $S^{(r)}$ is a $\g'$-submodule of $M$ for every $r$ (see Lemma \ref{lemma1MainCC});
\item[(2)] for any $u \in S^{(r+2)} \setminus S^{(r+1)}$ there exists $g \in U(\g')$ such that $g(u) \in S^{(r+1)} \setminus S^{(r)}$ (see Lemma \ref{lemma2MainCC});
\item [(3)] $S^{(r+1)} / S^{(r)}$ is a semisimple $\g'$-module and 
$$S^{(r+1)} / S^{(r)} \cong \bigoplus_{l+m+n=r} \overline{\soc}_{\g'}^{(l+1)} (\overline{\soc}_{\g_2}^{(m+1)}(\overline{\soc}_{\g_1}^{(n+1)} M))
$$ (see Lemma \ref{lemma3MainCC}).
\end{itemize}

Furthermore, if $N$ is a submodule of $M$, then Lemma \ref{lemma3MainCC} yields
\begin{align*}
(S^{(r+1)}  \cap N)/ (S^{(r)} \cap N) \cong \bigoplus_{l+m+n=r} \overline{\soc}_{\g'}^{(l+1)} (\overline{\soc}_{\g_2}^{(m+1)}(\overline{\soc}_{\g_1}^{(n+1)} N)).
\end{align*}
Thus the statement of the theorem follows.
\end{proof}

Below we give proofs of the key lemmas used in the proof of Theorem \ref{thmMainCC}. 

\begin{lemma} \label{lemma1MainCC} Let $\g'\subset \g_2 \subset \g_1 \subset \g$, $M$, and $S^{(r)}$ be as in Theorem \ref{thmMainCC}. Then $S^{(r)}$ is a $\g'$-submodule of $M$ for every $r$.
\end{lemma}

\begin{proof}
Let $u \in S^{(r)}$. Suppose that there exists $g \in \g'$ such that $g\cdot u\notin S^{(r)}$. In other words, there exist integers $l,m,n, n_1, n_2$ with $l+m+n = r$ and $n_1 + n_2 = n$ and index sets $i_1, \dots, i_{n_1}$, $j_1, \dots, j_{n_2}$, $I_1, \dots, I_m$, and $J_1, \dots, J_l$ such that
\begin{align} \label{eqNotin}
\tilde{\Phi}_{J_1, \dots, J_l} \circ \Phi_{I_1, \dots, I_m} \circ (L_{i_1, \dots, i_{n_1}} \otimes M_{j_1, \dots, j_{n_2}}) (g\cdot u) \neq 0.
\end{align}

We set $ u'= \Phi_{I_1, \dots, I_m} \circ (L_{i_1, \dots, i_{n_1}} \otimes M_{j_1, \dots, j_{n_2}}) (u)$. Since $L_{i_1, \dots, i_{n_1}} \otimes M_{j_1, \dots, j_{n_2}}$ and $\Phi_{I_1, \dots, I_m}$ are $\g'$-module homomorphisms, (\ref{eqNotin}) implies that  $u'\neq 0$. Furthermore, 
$\tilde{\Phi}_{J_1, \dots, J_l} (u') = 0$,  whereas $\tilde{\Phi}_{J_1, \dots, J_l} (g\cdot u') \neq 0$.
This can only happen if an element $x_{i}$ or $y_{i}$ appears in a monomial in $u'$ in one of the positions specified by $J_1, \dots, J_l$. Let without loss of generality $J_1 = (1,2)$ and let $u'$ have the form 
\begin{align} \label{eqFormU}
u'= x_{i} \otimes u'_2 \otimes u'_3 \otimes \dots \otimes u'_{s_1} \otimes u'^*_1 \otimes \dots \otimes u'^*_{s_2} + u'',
\end{align}
where $u'_2,\dots, u'_{s_1} \in V$ and $u'^*_{1}, \dots, u'^*_{s_2} \in V_*$. If $\g \cong \gl(\infty)$ then $s_1 =p-n_1-m$, $s_2 = q-n_2-m$. Otherwise, $s_1 = p-n-2m$, $s_2 = 0$. 
Then, $\tilde{\Phi}_{J_1, \dots, J_l} (u'') = 0$ and 
$$g\cdot u' = (g\cdot x_{i}) \otimes u'_2 \otimes \dots \otimes u'_{s_1} \otimes u'^*_1 \otimes \dots \otimes u'^*_{s_2} + x_{i}\otimes u''' + g\cdot u''.
$$ 
Moreover, $\tilde{\Phi}_{J_1, \dots, J_l} (x_{i}\otimes u''')= 0$. Assume at first that $\tilde{\Phi}_{J_1, \dots, J_l} (g\cdot u'') = 0$ as well. Then
\begin{align*}
\tilde{\Phi}_{J_1, \dots, J_l} (g\cdot u') = \tilde{\Phi}_{J_1} ((g\cdot x_{i})\otimes u'_2) \circ \tilde{\Phi}_{J'_2, \dots, J'_{l}}(u'_3\otimes \dots \otimes u'_{s_1} \otimes u'^*_1\otimes \dots \otimes u'^*_{s_2}) \neq 0.
\end{align*} 
Here, if $J_2 = (i,j)$ then $J'_2 = (i-2, j-2)$ and similarly for $J'_3, \dots, J'_l$.

Therefore, $\tilde{\Phi}_{J_2, \dots, J_{l}} (u') \neq 0$. The last inequality and (\ref{eqFormU}) imply that there exists an index $i_{n_1 +1}$ such that
\begin{align*}
\tilde{\Phi}_{J_2, \dots, J_l} \circ \Phi_{I_1, \dots, I_m} \circ (L_{i_1, \dots, i_{n_1}, i_{n_1 + 1}} \otimes M_{j_1, \dots, j_{n_2}}) ( u) \neq 0.
\end{align*}
This is a contradiction with $u \in S^{(r)}$.

If $\tilde{\Phi}_{J_1, \dots, J_l} (g\cdot u'') \neq 0$ we can replace $u'$ in the above discussion with $u''$, which has a strictly smaller number of monomials that $u'$. Thus in finitely many steps we will reach a contradiction with the choice of $u$, and this proves the statement.
\end{proof}

\begin{lemma} \label{lemma2MainCC} Let $\g'\subset \g_2 \subset \g_1 \subset \g$, $M$, and $S^{(r)}$ be as in Theorem \ref{thmMainCC}. Then for any $u \in S^{(r+2)} \setminus S^{(r+1)}$ there exists $g \in U(\g')$ such that $g \cdot u \in S^{(r+1)} \setminus S^{(r)}$.
\end{lemma}

\begin{proof}
We order the triples of numbers $(n,m,l)$ lexicographically. Let $u \in S^{(r+2)} \setminus S^{(r+1)}$ and let $(n,m,l)$ with $l+m+n = r+1$ be the largest triple for which there exist index sets $i_1, \dots, i_{n_1}$, $j_1, \dots, j_{n_2}$, $I_1, \dots, I_m$, and $J_1, \dots, J_l$ such that
\begin{align*}
\tilde{\Phi}_{J_1, \dots, J_l} \circ \Phi_{I_1, \dots, I_m} \circ (L_{i_1, \dots, i_{n_1}} \otimes M_{j_1, \dots, j_{n_2}}) (u) \neq 0.
\end{align*}

Suppose first that $n > 0$. Then at least one monomial in $u$ has an entry $x_t$ or $y_t$. Without loss of generality we may assume that $x_t$ appears in $u$. Then we take $g \in \g'$ of the form
\begin{align*}
g = \phi(v'_i \tilde{\otimes} v'^*_j),
\end{align*}
where $\phi$ denotes as before the embedding $\g' \hookrightarrow \g$.
We choose $i$ and $j$ such that $v^{*s}_{j}$ and $w^{*s}_i$ do not appear in $u$ for any $s$, at least one of them pairs non-degenerately with $x_t$, and they pair trivially with all basis elements from $V_2$ which appear in $u$. Furthermore, $v^s_i$ and $w^s_j$ do not appear in $u$ for any $s$ and pair trivially with all basis elements from $V_{2*}$ which appear in $u$. Then
\begin{align*}
\tilde{\Phi}_{J_1, \dots, J_l} \circ \Phi_{I_1, \dots, I_m} \circ (L_{i_1, \dots, i_{n_1}} \otimes M_{j_1, \dots, j_{n_2}}) (g\cdot u) = 0
\end{align*}
and there exist indices $i''_1, \dots, i''_{n_1 - 1}$ such that 
\begin{align*}
\tilde{\Phi}_{J_1, \dots, J_l} \circ \Phi_{I_1, \dots, I_m} \circ (L_{i''_1, \dots, i''_{n_1-1}} \otimes M_{j_1, \dots, j_{n_2}}) (g\cdot u) \neq 0.
\end{align*}

Now suppose that there exists another triple $(n',m',l')$ with $l'+m'+n' = r+1$ for which there are index sets $i'_1, \dots, i'_{n'_1}$, $j'_1, \dots, j'_{n'_2}$, $I'_1, \dots, I'_{m'}$, and $J'_1, \dots, J'_{l'}$ such that
\begin{align*}
\tilde{\Phi}_{J'_1, \dots, J'_{l'}} \circ \Phi_{I'_1, \dots, I'_{m'}} \circ (L_{i'_1, \dots, i'_{n'_1}} \otimes M_{j'_1, \dots, j'_{n'_2}}) (g\cdot u) \neq 0.
\end{align*}
But then there exists an index $i'_{n'+1}$ such that
\begin{align*}
\tilde{\Phi}_{J'_1, \dots, J'_{l'}} \circ \Phi_{I'_1, \dots, I'_{m'}} \circ (L_{i'_1, \dots, i'_{n'_1}, i'_{n'_1+1}} \otimes M_{j'_1, \dots, j'_{n'_2}}) (u) \neq 0,
\end{align*}
which contradicts with the choice of $u$. Hence, $g\cdot u \in S^{(r+1)} \setminus S^{(r)}$.

Now, suppose that $n = 0$. This means that $u \in \soc_{\g_1} M$ and $u$ consists only of elements from $V_1$ and $V_{1*}$. Notice that, when restricted to $\soc_{\g_1} M$, both maps $\tilde{\Phi}_{J_1, \dots, J_{l}}$ and $\Phi_{I_1, \dots, I_{m}}$ are $\g'$-module homomorphisms. Therefore, in this case the statement of the lemma reduces to the following claim. Let $l+m = r+1$, and $I_1, \dots, I_m$, and $J_1, \dots, J_l$ be such that
\begin{align*}
\tilde{\Phi}_{J_1, \dots, J_l} \circ \Phi_{I_1, \dots, I_m} (u) \neq 0.
\end{align*}
Then there exists $g \in U(\g')$ such that 
$\tilde{\Phi}_{J_1, \dots, J_l} \circ \Phi_{I_1, \dots, I_m} (g\cdot u) = 0$
and 
$\tilde{\Phi}_{J'_1, \dots, J'_{l'}} \circ \Phi_{I'_1, \dots, I'_{m'}} (g\cdot u) \neq 0$
for some $I'_1, \dots, I'_{m'}$, and $J'_1, \dots, J'_{l'}$ with $l'+m'= r$.
 
To prove this claim we consider two cases.
\begin{itemize}
\item[(1)] Let $
\tilde{\Phi}_{J_1, \dots, J_l} \circ \Phi_{I_1, \dots, I_m} (u) \neq 0$ with $l \geq 1$. Then we can apply Lemma \ref{lemmaAux1C} to the element $\Phi_{I_1, \dots, I_m} (u)$. This yields an element $g \in U(\g')$ and disjoint index pairs $J'_1, \dots, J'_{l-1}$ such that 
\begin{align*}
\tilde{\Phi}_{J'_1, \dots, J'_{l-1}} \circ \Phi_{I_1, \dots, I_m}(g\cdot u) \neq 0.
\end{align*}
\item[(2)]  Let $\tilde{\Phi}_0 \circ \Phi_{I_1, \dots, I_m} (u) \neq 0$. Then the statement follows from Lemma \ref{lemmaAux2C}.
\end{itemize}
\end{proof}

\begin{lemma} \label{lemmaAux1C} Let $u \in \soc_{\g_1} M$ be such that $\tilde{\Phi}_{J_1,\dots, J_l}(u) \neq 0$ for some $l$ and some $J_1, \dots, J_l$. 
Then there exists $g \in U(\g')$ such that $\tilde{\Phi}_{J_1,\dots, J_l}(g\cdot u) = 0$ and there exists a collection $J'_1, \dots, J'_{l-1}$ such that $\tilde{\Phi}_{J'_1,\dots, J'_{l-1}}(g\cdot u) \neq 0$.
\end{lemma}

\begin{proof}
Let $u = u_1 + \dots +u_t$ where without loss of generality
\begin{align*}
u_1 = u_1'\otimes u_1'' = u_1'\otimes z_{i_1}\otimes \dots \otimes z_{i_{s_1}} \otimes t_{j_1}\otimes \dots \otimes t_{j_{s_2}},
\end{align*}
where $u_1'$ contains only elements from $V_2$ and $V_{2*}$. Then either $\tilde{\Phi}_{J_1, \dots , J_l} (u_1) = 0$ or $\tilde{\Phi}_{J_1, \dots , J_l} (u_1) = \tilde{\Phi}_{J_1, \dots , J_l} (u_1') \otimes u_1''$. The statement for $\tilde{\Phi}_{J_1, \dots , J_l} (u_1')$ follows from Proposition \ref{propGenEmb3}.
\end{proof}

\begin{lemma} \label{lemmaAux2C}
Let $u \in \soc_{\g_1} M$ be such that $\Phi_{I_1, \dots, I_{m}}(u) \neq 0$ for some $m > 0$ and some $I_1, \dots, I_m$. Then there exists $g \in U(\g')$ such that $\Phi_{I_1, \dots, I_{m}}(g\cdot u) = 0$ and $\Phi_{I'_1, \dots, I'_{m-1}}(g\cdot u) \neq 0$ for some collection $I'_1, \dots, I'_{m-1}$.
\end{lemma}

\begin{proof}
The proof is analogous to the second part of the proof of Proposition \ref{propEmb2Appr1}. Let $u$ be as in the statement of the lemma. Then $u = u_1 + \dots +u_h$ and without loss of generality we can fix $I_1, \dots, I_m$ such that $u_1$ has the form
\begin{align*}
u_1 = &z_{i_1}\otimes z_{i_2}\otimes \dots \otimes z_{i_{m}} \otimes u'_1\otimes  t_{i_1}\otimes t_{i_2}\otimes \dots \otimes t_{i_{m}} - \\
&v^n_j\otimes z_{i_2}\otimes \dots \otimes z_{i_{m}} \otimes u'_1\otimes  v^{n*}_j\otimes t_{i_2}\otimes \dots \otimes t_{i_{m}} - \dots -\\
&z_{i_1}\otimes z_{i_2}\otimes \dots \otimes z_{i_{m}}\otimes v^{n}_j \otimes u'_1\otimes  t_{i_1}\otimes t_{i_2}\otimes \dots \otimes t_{i_{m}}\otimes v^{n*}_j + u''_1,
\end{align*}
for some $j \in I$ and some $n = 1, \dots, k$. Moreover, $\Phi_{I_1, \dots, I_{m}}(u''_1) = 0$. The elements $u_2, \dots, u_h$ have similar form.
Notice that if the basis vectors $v^{s}_j$ and $w^{s*}_j$, $s = 1, \dots, k$, appear in total at most $t$ times in any monomial in $u'_1$ then they appear at most $t+1$ times in any monomial in $v^n_j\otimes z_{i_2}\otimes \dots \otimes z_{i_{m}} \otimes u'_1\otimes  v^{n*}_j\otimes t_{i_2}\otimes \dots \otimes t_{i_{m}}$. Let us take $g_i \in \g'$ of the form
$$
g_i = \phi(v'_i \tilde{\otimes} v'^*_j),
$$
where $\phi$ denotes again the embedding $\g' \subset \g$. 
Then, if $v^{s}_{i_1}, \dots, v^{s}_{i_{t+1}}$ and $w^{s*}_{i_1}, \dots, w^{s*}_{i_{t+1}}$, $s = 1, \dots, k$, are vectors that do not appear at all in the expression of $u$, then $
(g_{i_1}\circ \dots \circ g_{i_{t+1}})(u_1)
$
has the desired properties.
We proceed in the same way with $u_2, \dots, u_h$ to prove the statement.
\end{proof}

\begin{lemma} \label{lemma3MainCC} Let $\g' \subset \g_2 \subset \g_1 \subset \g$, $M$, and $S^{(r)}$ be as in Theorem \ref{thmMainCC}. Then $S^{(r+1)} / S^{(r)}$ is a semisimple $\g'$-module and 
$$S^{(r+1)} / S^{(r)} \cong \bigoplus_{l+m+n=r} \overline{\soc}_{\g'}^{(l+1)} (\overline{\soc}_{\g_2}^{(m+1)}(\overline{\soc}_{\g_1}^{(n+1)} M)).
$$

Furthermore, if $N$ is a submodule of $M$ then 
\begin{align*}
(S^{(r+1)}  \cap N)/ (S^{(r)} \cap N) \cong \bigoplus_{l+m+n=r} \overline{\soc}_{\g'}^{(l+1)} (\overline{\soc}_{\g_2}^{(m+1)}(\overline{\soc}_{\g_1}^{(n+1)} N)).
\end{align*}
\end{lemma}

\begin{proof}
Let $S^{(n+1, m+1, l+1)}$ denote the $\g'$-submodule of $M$ of elements $v$ with the following properties: 
\begin{itemize}
\item [(1)] $v \in \soc^{(n+1)}_{\g_1} M$;
\item [(2)] $\pi_n(v) \in \soc^{(m+1)}_{\g_2}(\overline{\soc}^{(n+1)}_{\g_1} M)$, where $\pi_n : \soc^{(n+1)}_{\g_1} M \rightarrow \overline{\soc}^{(n+1)}_{\g_1} M$;
\item [(3)] $\pi_{mn}\circ \pi_n(v) \in \soc_{\g'}^{(l+1)} (\overline{\soc}_{\g_2}^{(m+1)}(\overline{\soc}_{\g_1}^{(n+1)} M))$, where $$\pi_{mn} : \soc^{(m+1)}_{\g_2}(\overline{\soc}^{(n+1)}_{\g_1} M) \rightarrow \overline{\soc}^{(m+1)}_{\g_2}(\overline{\soc}^{(n+1)}_{\g_1} M).$$
\end{itemize}
Notice that $S^{(n'+1, m'+1, n'+1)} \subset S^{(n+1, m+1, l+1)}$ if and only if $(n',m',l') < (n,m,l)$ in the lexicographic order.
Thus, we obtain a filtration on M
\begin{align*}
&0 \subset S^{(1,1,1)} \subset S^{(1,1,2)} \subset \dots \subset S^{(1,1,L_{31})} \subset S^{(1,2,1)} \subset \dots \subset S^{(1, L_{21}, L_{32})} \subset \\
&S^{(2,1, 1)} \subset \dots \subset S^{(L_1, L_2, L_3)},
\end{align*}
where $L_1$ is the Loewy length of $M$ as a $\g_1$-module and the other $L$'s denote the Loewy lengths of the respective modules. 
Next, we take a coarser filtration in which only elements $S^{(n+1, m+1, l+1)}$ with $l+m+n = r$ appear and we intersect this filtration with $S^{(r+1)}$:
\begin{equation*} 
\begin{aligned}
&0 \subset S^{(1,1, r+1)} \cap S^{(r+1)} \subset S^{(1,2, r)} \cap S^{(r+1)} \subset \dots \subset S^{(1, r+1, 1)} \cap S^{(r+1)} \subset \\ 
&S^{(2,1,r)} \cap S^{(r+1)} \subset \dots \subset S^{(2, r, 1)} \cap S^{(r+1)} \subset \dots \subset S^{(r+1,1, 1)} \cap S^{(r+1)}.
\end{aligned}
\end{equation*}

We build now the corresponding filtration on the quotient $S^{(r+1)} / S^{(r)}$.
\begin{align} \label{eqFiltrC}
0 \subset (S^{(1,1,r+1)} \cap S^{(r+1)} + S^{(r)}) / S^{(r)} \subset \dots \subset (S^{(r+1,1,1)} \cap S^{(r+1)} + S^{(r)}) / S^{(r)}.
\end{align}

For any $l$, $m$, and $n$, we define the following maps:
\begin{align*}
K_n = \bigoplus_{n_1 + n_2 = n} \bigoplus_{\substack{i_1 <\dots < i_{n_1} \\ j_1 < \dots < j_{n_2}}} L_{i_1,\dots, i_{n_1}}\otimes M_{j_1, \dots, j_{n_2}};
\end{align*}
\begin{align*}
K'_m = \bigoplus_{\{I_1, \dots, I_m\}} \Phi_{I_1, \dots, I_m};
\end{align*}
\begin{align*}
K''_l = \bigoplus_{\{J_1, \dots, J_l\}} \tilde{\Phi}_{J_1, \dots, J_l}.
\end{align*}
Then $K_0 = K'_0 = K''_0 = \id$. Moreover, $K_n$ and $K'_m$ are $\g'$-module homomorphisms for any $n$, whereas $K''_l$ for $l > 0$ is just a linear map.

Proposition \ref{propEmb1} implies that $K_n(\soc_{\g_1}^{(n+1)} M) \cong \overline{\soc}_{\g_1}^{(n+1)} M$. 
Furthermore, Propositions \ref{propEmb2Appr1} and \ref{propGenEmb3} imply the following $\g'$-module isomorphism
\begin{align*} 
K''_l \circ K'_m \circ K_n(S^{(n+1, m+1, l+1)}) \cong  \overline{\soc}_{\g'}^{(l+1)} (\overline{\soc}_{\g_2}^{(m+1)}(\overline{\soc}_{\g_1}^{(n+1)} M)).
\end{align*}
Thus, in what follows for $n+m+l = r$ we can consider $K''_l \circ K'_m \circ K_n$ as a map
\begin{align*} 
K''_l \circ K'_m \circ K_n :S^{(n+1, m+1, l+1)} \cap S^{(r+1)} \rightarrow  \overline{\soc}_{\g'}^{(l+1)} (\overline{\soc}_{\g_2}^{(m+1)}(\overline{\soc}_{\g_1}^{(n+1)} M)).
\end{align*}

To prove that $S^{(r+1)} / S^{(r)}$ is a semisimple $\g'$-module we proceed by induction on the elements of the filtration (\ref{eqFiltrC}). First, Proposition \ref{propGenEmb3} implies that
\begin{align*}
K''_{r} \circ K'_0 \circ K_0 : (S^{(1,1,r+1)} \cap S^{(r+1)} + S^{(r)}) / S^{(r)} \rightarrow \overline{\soc}_{\g'}^{(r+1)} ({\soc}_{\g_2}({\soc}_{\g_1}M))
\end{align*}
is an isomorphism of $\g'$-modules. 
Hence $(S^{(1,1,r+1)} \cap S^{(r+1)} + S^{(r)}) / S^{(r)}$ is semisimple.

Next, suppose that $S^{(n'+1,m'+1,l'+1)} \cap S^{(r+1)} + S^{(r)}) / S^{(r)}$ is semisimple for some $n'+m'+l'= r$ and that 
\begin{equation} \label{eqInd1}
\begin{aligned}
&(S^{(n'+1,m'+1,l'+1)} \cap S^{(r+1)} + S^{(r)}) / S^{(r)} \cong\\
&\bigoplus_{\substack{l''+m''+n'' = r \\ (n'', m'', l'') \leq (n', m', l')}} \overline{\soc}_{\g'}^{(l''+1)} (\overline{\soc}_{\g_2}^{(m''+1)}(\overline{\soc}_{\g_1}^{(n''+1)} M)).
\end{aligned}
\end{equation}

Let $(n,m,l)$ be the immediate successor of $(n',m',l')$ in the lexicographic order of triples of integers with sum equal to $r$. We prove next that $(S^{(n+1,m+1,l+1)} \cap S^{(r+1)} + S^{(r)}) / S^{(r)}$ is semisimple.

Take an element $u \in (S^{(n+1,m+1,l+1)} \cap S^{(r+1)} + S^{(r)}) / S^{(r)}$ of the form $u = u_1 + \dots + u_s + S^{(r)}$. Let without loss of generality
\begin{align*}
u_1 = x_{i_1}\otimes \dots\otimes x_{i_{n_1}} \otimes u_1'\otimes y_{j_1}\otimes \dots \otimes y_{j_{n_2}} + u_1'',
\end{align*}
where $n_1 + n_2 = n$ and $u''_1$ has less than $n$ elements $x_j$ and $y_j$. Furthermore,  
\begin{align*}
u_1'= z_{s_1}\otimes \dots \otimes z_{s_m} \otimes u_1''' \otimes t_{s_1}\otimes \dots \otimes t_{s_m} +u_1^{(iv)},
\end{align*}
where $u_1^{(iv)}$ has less that $m$ terms of the form $z_i \otimes t_i$.
Let the elements $u_2, \dots, u_s$ have a similar form. Then for any $g\in \g'$ either $g\cdot u = 0$ or $g\cdot u \notin S^{(n'+1,m'+1,l'+1)} \cap S^{(r+1)} + S^{(r)}) / S^{(r)}$. Let $U$ denote the submodule of $(S^{(n+1,m+1,l+1)} \cap S^{(r+1)} + S^{(r)}) / S^{(r)}$ generated by all such elements $u$. Then
\begin{align} \label{eqInd2}
(S^{(n+1,m+1,l+1)} \cap S^{(r+1)} + S^{(r)}) / S^{(r)} = (S^{(n'+1,m'+1,l'+1)} \cap S^{(r+1)} + S^{(r)}) / S^{(r)} \oplus U.
\end{align}

Moreover, we claim that
\begin{align} \label{eqInd3}
K''_l \circ K'_m \circ K_n : (S^{(n+1, m+1, l+1)} \cap S^{(r+1)} + S^{(r)}) / S^{(r)} \rightarrow \overline{\soc}_{\g'}^{(l+1)} (\overline{\soc}_{\g_2}^{(m+1)}(\overline{\soc}_{\g_1}^{(n+1)} M))
\end{align}
is a well-defined surjective homomorphism of $\g'$-modules with kernel $(S^{(n'+1, m'+1, l'+1)} \cap S^{(r+1)} + S^{(r)}) / S^{(r)}$.

It is clear that the above map is well-defined. Moreover, it is a $\g'$-module homomorphism since $K''_l$ is a $\g'$-module homomorphism when restricted to $K'_m \circ K_n((S^{(n+1, m+1, l+1)})$. 
To prove surjectivity we notice that
for every $u \neq 0$ from the right-hand side in (\ref{eqInd3}) there exists $v \in S^{(n+1, m+1, l+1)}$ such that $K''_l \circ K'_m \circ K_n(v) = u$ and 
\begin{itemize}
\item $v \in \soc^{(n+1)}_{\g_1} M \setminus \soc^{(n)}_{\g_1} M$,
\item $\pi_n(v) \in \soc_{\g_2}^{(m+1)}(\overline{\soc}_{\g_1}^{(n+1)} M) \setminus \soc_{\g_2}^{(m)}(\overline{\soc}_{\g_1}^{(n+1)} M)$,
\item $\pi_{mn}(v) \in \soc_{\g'}^{(l+1)} (\overline{\soc}_{\g_2}^{(m+1)}(\overline{\soc}_{\g_1}^{(n+1)} M)) \setminus \soc_{\g'}^{(l)} (\overline{\soc}_{\g_2}^{(m+1)}(\overline{\soc}_{\g_1}^{(n+1)} M))$.
\end{itemize}
But then $v \in S^{(n+1, m+1, l+1)} \cap S^{(r+1)}$. This proves surjectivity. 

The last step is to compute the kernel of the map $K''_l \circ K'_m \circ K_n$ from (\ref{eqInd3}).
Suppose first that $m \neq 0$. 
Then the immediate predecessor of $(n,m,l)$ in the lexicographic order is $(n, m-1, l+1)$. Thus we need to prove that 
\begin{align} \label{eqKer1}
\ker K''_l \circ K'_m \circ K_n  = (S^{(n + 1, m, l + 2)} \cap S^{(r+1)} + S^{(r)}) / S^{(r)}.
\end{align}
Clearly, $\ker K''_l \circ K'_m \circ K_n \supseteq (S^{(n + 1, m, l + 2)} \cap S^{(r+1)} + S^{(r)}) / S^{(r)}$.
To prove the opposite inclusion, let  $[v] \in \ker K''_l \circ K'_m \circ K_n$. Suppose first that $K_n(v) = 0$. Then for some $m_1$ and $l_1$, $v \in S^{(n, m_1 +1, l_1 + 1)} \cap S^{(r+1)} \subset S^{(n + 1, m, l + 2)} \cap S^{(r+1)}$. Next, let $v$ be such that $K_n(v) \neq 0$ and $K'_m(K_n(v)) = 0$. But then $v \in S^{(n+1, m, l+2)} \cap S^{(r+1)}$. Finally, if $K'_m(K_n(v)) \neq 0$, 
it follows that $v \in S^{(r)}$. This proves (\ref{eqKer1}).

Now, suppose that $m =  0$ and $n \neq 0$. 
Then the immediate predecessor of $(n,0,l)$ is $(n-1, l+1, 0)$. Moreover, $(S^{(n, l+2, 1 )} \cap S^{(r+1)} + S^{(r)} )/ S^{(r)}\subseteq \ker K''_l \circ K'_0 \circ K_n$. And vice versa, if $[v] \in \ker K''_l \circ K'_0 \circ K_n$ is such that $K_n(v) = 0$ then $v \in S^{(n, m_1 +1, l_1 + 1)} \cap S^{(r+1)} \subset S^{(n, l+2, 1)} \cap S^{(r+1)}$. Hence,
$$\ker K''_l \circ K'_0 \circ K_n  = (S^{(n, l+2, 1)} \cap S^{(r+1)} + S^{(r)}) / S^{(r)}.$$

Thus, we proved (\ref{eqInd3}). Moreover, (\ref{eqInd1}), (\ref{eqInd2}), and (\ref{eqInd3}) imply
\begin{align*}
(S^{(n+1,m+1,l+1)} \cap S^{(r+1)} + S^{(r)}) / S^{(r)} \cong \bigoplus_{\substack{n'+m'+l'= r \\ (n',m', l') \leq (n,m,l) }} \overline{\soc}_{\g'}^{(l'+1)} (\overline{\soc}_{\g_2}^{(m'+1)}(\overline{\soc}_{\g_1}^{(n'+1)} M)).
\end{align*}
 
This holds for every element in the filtration (\ref{eqFiltrC}), so in particular $S^{(r+1)} / S^{(r)}$ is semisimple and 
\begin{align*}
S^{(r+1)} / S^{(r)} \cong \bigoplus_{n+m+l= r } \overline{\soc}_{\g'}^{(l+1)} (\overline{\soc}_{\g_2}^{(m+1)}(\overline{\soc}_{\g_1}^{(n+1)} M)).
\end{align*}

Note that if $N \subset M$ is any submodule of $M$, then we can interesect the above filtrations with $N$ and obtain 
\begin{align*}
(S^{(r+1)} \cap N)/ (S^{(r)} \cap N) \cong \bigoplus_{l+m+n = r} \overline{\soc}_{\g'}^{(l+1)} (\overline{\soc}_{\g_2}^{(m+1)}(\overline{\soc}_{\g_1}^{(n+1)} N)).
\end{align*}
\end{proof}

Theorem \ref{thmMainCC} shows that we can reduce the branching problem for embeddings of general tensor type to branching problems for embeddings of types I, II, and III. In the next section we will explicitly compute the branching laws for embeddings of $\gl(\infty) \subset \gl(\infty)$ of types I, II, and III and for any simple tensor $\gl(\infty)$-module $V_{\lambda, \mu}$. 

\section{Branching laws for embeddings of $\gl(\infty)$ into $\gl(\infty)$} \label{secGL}

\subsection{Embeddings of type I} \label{secGLEmb1}

Let $\g', \g \cong \gl(\infty)$ and let $\g' \subset \g$ satisfy the conditions 
\begin{equation} \label{eqCond1}
\begin{aligned}
\soc_{\g'} V \cong V' \oplus N_a, && V/\soc_{\g'} V \cong N_b, \\
\soc_{\g'} V_* \cong V'_* \oplus N_c, && V_*/\soc_{\g'} V_* \cong N_d. 
\end{aligned}
\end{equation}
So far we make no assumptions about the pairing between $N_a$ and $N_c$.
First we extend the definition of the Gelfand-Tsetlin multiplicity $m^k_{\lambda, \sigma}$ to the case $k = \infty$. More precisely, we define
$$
m^{\infty}_{\lambda, \sigma} = \lim_{k \rightarrow \infty} m^{k}_{\lambda, \sigma}.
$$
In other words, 
\begin{align*}
m^{\infty}_{\lambda, \sigma} = \left\{ \begin{array}{ll}
         0 & \text{if } m^{k}_{\lambda, \sigma}=0 \text{ for all } k,\\
         1 & \text{if } \sigma = \lambda, \text{ i.e. } m^{k}_{\lambda, \sigma}=1 \text{ for all } k,\\
        \infty & \text{if }m^{k}_{\lambda, \sigma}>1 \text{ for at least one } k. \end{array} \right.
\end{align*}
Furthermore, for convenience we set $m^0_{\lambda, \sigma} = 0$ for $\lambda \neq \sigma$ and $m^0_{\lambda, \lambda} = 1$. We will refer to the values $m^k_{\lambda, \sigma}$ for $k \in \ZZ_{\geq 0} \sqcup \{\infty\}$ as the \textit{extended Gelfand-Tsetlin multiplicities}. 

We start with determining the socle filtration of any simple tensor $\g$-module $V_{\lambda, 0} \subset V^{\otimes p}$. Before stating the general result we need two lemmas.
\begin{lemma} \label{lemma1_embed} 
Let $\g' \subset \g$ satisfy (\ref{eqCond1}) with $a = 0$. Then for any $V_{\lambda, 0} \subset V^{\otimes p}$ we have
\begin{align*}
&\soc_{\g'} V_{\lambda, 0} = V'_{\lambda, 0},\\
&\overline{\soc}_{\g'}^{(r+1)} V_{\lambda, 0} \cong \bigoplus_{|\lambda'| = |\lambda| - r} m^{b}_{\lambda, \lambda'} V'_{\lambda', 0},
\end{align*}
where $m^b_{\lambda, \lambda'}$ are the extended Gelfand-Tsetlin multiplicities.
\end{lemma} 
\begin{proof}

Proposition \ref{propEmb1} and the remark after it imply that 
\begin{align} \label{eqEmb1Vp}
\overline{\soc}_{\g'}^{(r+1)}V^{\otimes p} \cong \bigoplus_{i_1 < \dots < i_r} N_b^{\otimes r}\otimes (V' \oplus N_a)^{\otimes p-r}.
\end{align}
Thus in the case $a=0$ it follows from (\ref{eqEmb1Vp}) and Theorem 2.1 in \cite{PSt} that
$$
\overline{\soc}_{\g'}^{(r+1)}V^{\otimes p} \cong \binom{p}{r}N_b^{\otimes r}\otimes (V')^{\otimes p-r} \cong \bigoplus_{|\lambda'| = p-r} c_{\lambda'} V'_{\lambda', 0}
$$
for some multiplicities $c_{\lambda'}$. Moreover, property (\ref{eqPropSoc1}) of socle filtrations implies that 
$\overline{\soc}_{\g'}^{(r+1)} V_{\lambda, 0} \subset \overline{\soc}_{\g'}^{(r+1)}V^{\otimes p}$, hence 
\begin{align} \label{eqSecSubm2} 
\overline{\soc}_{\g'}^{(r+1)} V_{\lambda, 0} \cong \bigoplus_{|\lambda'| = p-r} c'_{\lambda'} V'_{\lambda', 0}
\end{align}
for some unknown $c'_{\lambda'}$. Thus, we only need to compute the multiplicity with which each $V'_{\lambda', 0}$ enters the decomposition of $V_{\lambda, 0}$. Note that on distinct layers of the socle filtration non-isomorphic simple constituents $V'_{\lambda', 0}$ appear.

Let $\{v'_i\}_{i \in \ZZ_{>0}}$ and $\{v'^*_i\}_{i \in \ZZ_{>0}}$ be as before a pair of dual bases in $V'$ and $V'_*$, and let $\{\xi_i\}_{i \in \ZZ_{>0}}$, $\{\xi^*_i\}_{i \in \ZZ_{>0}}$ be a pair of dual bases in $V$ and $V_*$. For each $n$, put $V_n = \mathrm{span} \{\xi_{1}, \dots, \xi_{n} \}$ and $V_n^* = \mathrm{span} \{\xi_{1}^*, \dots, \xi_{n}^* \}$. The pairing between $V$ and $V_*$ restricts to a non-degenerate pairing between $V_n$ and $V_n^*$. Therefore we can define the Lie algebra $\g_n = V_n \otimes V_n^*$. Furthermore, we set $\h_n = \h_{\g} \cap \g_n$ and $\b_n = \b_{\g} \cap \g_n$. It is clear that $\g_n \cong \gl(n)$ and that $\h_n$ (respectively, $\b_n$) is a Cartan (respectively, Borel) subalgebra of $\g_n$. Moreover, if we set $V_{\lambda, 0}^n = V_{\lambda, 0} \cap V_n^{\otimes p}$, then for $n \geq p$, $V_{\lambda, 0}^n $ is a highest weight $\g_n$-module with highest weight $(\lambda,0)$ with respect to $\b_n$. 

Now we apply the same procedure to $\g'$. We define $V'_n = \mathrm{span} \{v'_{1}, \dots, v'_{n} \}$ and $V_n'^{*} = \mathrm{span} \{v'^*_{1}, \dots, v'^*_{n} \}$. We set $\g'_n = V'_n \otimes V_n'^{*}$, $\h'_n = \h_{\g'} \cap \g'_n$, $\b'_n = \b_{\g'} \cap \g'_n$, and $V'^n_{\lambda, 0} = V'_{\lambda, 0} \cap V'^{\otimes p}_n$. In this way we obtain a commutative diagram of inclusions
\begin{equation*}
\xymatrix{
&\g'_1\ar[d]\ar[r]&\g'_{2}\ar[d]\ar[r]&\dots \ar[r] &\g'_k \ar[d]\ar[r] &\dots \ar[r]&\g' \ar[d]\\
&\g_{m_1}\ar[r]&\g_{m_2}\ar[r]&\dots \ar[r] &\g_{m_k} \ar[r] &\dots \ar[r]&\g\\
}
\end{equation*}
in which all horizontal arrows are standard inclusions. 
Then for large $k$ each embedding $\g'_k \rightarrow \g_{m_k}$ is diagonal, i.e. in our case
\begin{align*}
V_{m_k} \cong V'_k \oplus N'_k,
\end{align*} 
where $N'_k$ is some finite-dimensional trivial $\g'_k$-module. Moreover, if we set $n_k = \dim N'_k = \codim_{V_{m_k}} V'_k$ then $b = \codim_{V}V' = \lim_{k\rightarrow \infty} n_k$. Thus, when $b$ is finite we obtain that for large $k$ all vertical embeddings in the above diagram are of signature $(1,0,b)$. In the case $b = \infty$, for large $k$ all vertical embeddings are of signature $(1,0,n_k)$ with $\lim_{k\rightarrow \infty} n_k = \infty$.

Let us consider first the case when $b$ is finite. Then for large $k$ we can use the Gelfand-Tsetlin rule for the embedding $\g'_k \rightarrow \g_{m_k}$ and the module $V^{m_k}_{\lambda, 0} = V_{\lambda, 0} \cap V_{m_k}^{\otimes p}$ to obtain the decomposition
\begin{align} \label{eqSecSubm1}
V^{m_k}_{\lambda, 0} \cong \bigoplus_{\lambda'} m^{b}_{\lambda, \lambda'} V'^{k}_{\lambda', 0}.
\end{align}

Then (\ref{eqSecSubm2}) and (\ref{eqSecSubm1}) imply 
\begin{align*}
({\soc}^{(r+1)}V_{\lambda, 0}) \cap V_{m_k}^{\otimes p}/({\soc}^{(r)}V_{\lambda, 0}) \cap V_{m_k}^{\otimes p} \cong \bigoplus_{|\lambda'| = p-r} m^{b}_{\lambda, \lambda'} V'^{k}_{\lambda', 0},
\end{align*}
and passing to the direct limit we obtain the statement.

Next, let $b = \infty$. Then from the Gelfand-Tsetlin rule we obtain
\begin{align*}
({\soc}^{(r+1)}V_{\lambda, 0}) \cap V_{m_k}^{\otimes p}/({\soc}^{(r)}V_{\lambda, 0} \cap V_{m_k}^{\otimes p}) \cong \bigoplus_{|\lambda'| = p-r} m^{n_k}_{\lambda, \lambda'} V'^{k}_{\lambda', 0},
\end{align*}
and passing to the direct limit we obtain the statement.
\end{proof}

\begin{lemma} \label{lemma2_embed} Let $\g' \subset \g$ satisfy (\ref{eqCond1}) with $b = 0$. Then any $V_{\lambda, 0} \subset V^{\otimes p}$ is a completely reducible $\g'$-module and 
$$
V_{\lambda, 0} \cong \bigoplus_{\lambda'} m^{a}_{\lambda, \lambda'} V'_{\lambda', 0},
$$
where $m^{a}_{\lambda, \lambda'}$ are again the extended Gelfand-Tsetlin multiplicities.
\end{lemma}

\begin{proof}
When $b = 0$ the map $f$ from (\ref{eqSES1}) is just the zero homomorphism, hence Proposition \ref{propEmb1} implies that $V_{\lambda, 0}$ is completely reducible. Therefore, we only need to compute the multiplicity of each $V'_{\lambda', 0}$ in the expression of $V_{\lambda, 0}$, and this is done in the same way as in Lemma \ref{lemma1_embed}.
\end{proof}

Now we can state the branching rule for $V_{\lambda, 0}$ for arbitrary values of $a$ and $b$.

\begin{thm} \label{thmVlambda}
Let $\g' \subset \g$ be an embedding which satisfies (\ref{eqCond1}). Then for any $V_{\lambda, 0} \subset V^{\otimes p}$
$$
\overline{\soc}_{\g'}^{(r+1)} V_{\lambda, 0} \cong \bigoplus_{\lambda''} \bigoplus_{|\lambda'| = |\lambda''| - r}m^{a}_{\lambda, \lambda''} m^{b}_{\lambda'', \lambda'} V'_{\lambda', 0}.
$$
\end{thm}

\begin{proof}
Let $\{v'_i\}$ and $\{v'^*_i\}$ be as before a pair of dual bases in $V'$ and $V'_*$. Let
\begin{align*}
&V = \mathrm{span} \{x_i\}_{i \in I_b} \cup \{z_i\}_{i \in I_a} \cup \{v'_i\}_{i \in I}, \\
&V_* = \mathrm{span} \{f_i\}_{i \in I_c \cup I_d} \cup \{v'^*_i\}_{i \in I},
\end{align*}
where $x_i$ and $z_i$ are as in Proposition \ref{propBasis} and $f_i$ is a collective notation for $y_i$ and $t_i$. We divide the $f_i$'s into two groups: those which pair non-degenerately with $\{z_i\}_{i \in I_a}$ we denote by $f'_i$, and the remaining ones we denote by $f''_i$. Next we define the subspaces
\begin{align*}
&V'' = \mathrm{span} \{x_i\}_{i \in I_b} \cup \{v'_i\}_{i \in I}, \\
&V''_* = \mathrm{span} \{f''_i\} \cup \{v'^*_i\}_{i \in I}.
\end{align*}
Then $V''$ and $V''_*$ pair non-degenerately and we set $\g'' = V'' \otimes V''_*$. 
The embedding $\g'' \subset \g$ satisfies the conditions of Lemma \ref{lemma2_embed}, hence
\begin{align} \label{eqEmbB}
V_{\lambda, 0} \cong \bigoplus_{\lambda''} m^{a}_{\lambda, \lambda''} V''_{\lambda'', 0},
\end{align}
where by $V''_{\lambda'', 0}$ we denote the simple tensor $\g''$-modules.

To see now how $V_{\lambda,0}$ decomposes over $\g'$ it is enough to see how each $V''_{\lambda'', 0}$ decomposes over $\g'$. 
The embedding $\g' \subset \g''$ satisfies the conditions of Lemma \ref{lemma1_embed}, hence
\begin{align} \label{eqEmbA}
\overline{\soc}_{\g'}^{(r+1)} V''_{\lambda'', 0} \cong \bigoplus_{|\lambda'| = |\lambda''| - r} m^{b}_{\lambda'', \lambda'} V'_{\lambda', 0}.
\end{align}
Then (\ref{eqEmbB}) and (\ref{eqEmbA}) imply the statement of the theorem.
\end{proof}

An analogous statement holds for submodules of $V_*^{\otimes q}$. Here is the result.

\begin{thm} \label{thmVmu} Let $\g' \subset \g$ be an embedding which satisfies (\ref{eqCond1}). Then for any $V_{0, \mu} \subset V_*^{\otimes q}$
$$
\overline{\soc}_{\g'}^{(r+1)} V_{0, \mu} \cong \bigoplus_{\mu''} \bigoplus_{|\mu'| = |\mu''| - r}m^{c}_{\mu, \mu''} m^{d}_{\mu'', \mu'} V'_{0, \mu'}.
$$
\end{thm}

Now we are ready to state the theorem for any simple $V_{\lambda, \mu}$.

\begin{thm} \label{thmEmb1} Let $\g' \subset \g$ be an embedding of type I, i.e. which satisfies (\ref{eqCond1}) and such that $N_a$ and $N_c$ pair trivially. Then for any simple $\g$-module $V_{\lambda, \mu} \subset V^{\{p,q\}}$
\begin{align*}
& \soc_{\g'}^{(r+1)} V_{\lambda, \mu} = (\sum_{n_1+n_2 = r}(\soc_{\g'}^{(n_1+1)} V_{\lambda, 0}) \otimes (\soc_{\g'}^{(n_2+1)} V_{0, \mu})) \cap V^{\{p,q\}}.
\end{align*}
Moreover,
\begin{align*}
\overline{\soc}_{\g'}^{(r+1)}V_{\lambda, \mu} \cong \bigoplus_{n_1+n_2=r}\bigoplus_{\lambda'', \mu''} \bigoplus_{\substack{|\lambda'| = |\lambda''| - n_1 \\ |\mu'| = |\mu''| - n_2}} m^{a}_{\lambda, \lambda''}m^b_{\lambda'', \lambda'}m^{c}_{\mu, \mu''} m^d_{\mu'', \mu'} V'_{\lambda', \mu'}.
\end{align*}
\end{thm}

\begin{proof}
Note that 
\begin{align*}
&(\sum_{n_1+n_2 = r}(\soc_{\g'}^{(n_1+1)} V^{\otimes p}) \otimes (\soc_{\g'}^{(n_2+1)} V_*^{\otimes q}))\cap (V_{\lambda,0} \otimes V_{0, \mu}) = \\
& \sum_{n_1+n_2 = r}((\soc_{\g'}^{(n_1+1)} V_{\lambda,0}) \otimes (\soc_{\g'}^{(n_2+1)} V_{0, \mu})).
\end{align*}
Therefore, Corollary \ref{corSocEmb1} and the relation $V_{\lambda, \mu} = V^{\{p,q\}} \cap (V_{\lambda,0} \otimes V_{0, \mu})$ from \cite{PSt} imply
\begin{align*}
&\soc_{\g'}^{(r+1)} V_{\lambda, \mu} = (\soc_{\g'} ^{(r+1)}V^{\{p,q\}}) \cap V_{\lambda, \mu} = \\
&(\sum_{n_1+n_2 = r}(\soc_{\g'}^{(n_1+1)} V_{\lambda, 0}) \otimes (\soc_{\g'}^{(n_2+1)} V_{0,\mu})) \cap V^{\{p,q\}},
\end{align*}
which proves the first part of the statement. To compute the layers of the socle filtration we notice the following. If $\{A_{n_1}\}$ and $\{B_{n_2}\}$ are families of modules over a ring or an algebra such that $A_{n_1} \subset A_{n_1+1}$ for all $n_1$ and $B_{n_2} \subset B_{n_2+1}$ for all $n_2$, then
\begin{align*}
&(\sum_{n_1+n_2 = r} A_{n_1} \otimes B_{n_2}) / (\sum_{n_1+n_2 = r-1} A_{n_1} \otimes B_{n_2}) \cong
\bigoplus_{n_1+n_2 = r} A_{n_1}/A_{n_1-1} \otimes B_{n_2}/B_{n_2-1}.
\end{align*}

In our case, if we set $A_{n_1} = \soc_{\g'}^{(n_1+1)} V_{\lambda, 0}$ and $B_{n_2} = \soc_{\g'}^{(n_2+1)} V_{0, \mu}$ we obtain
\begin{equation} \label{eqVlambdaVmu}
\begin{aligned}
&(\sum_{n_1+n_2 = r}(\soc_{\g'}^{(n_1+1)} V_{\lambda, 0}) \otimes (\soc_{\g'}^{(n_2+1)} V_{0, \mu})) / (\sum_{n_1+n_2 = r-1}(\soc_{\g'}^{(n_1+1)} V_{\lambda, 0}) \otimes (\soc_{\g'}^{(n_2+1)} V_{0, \mu})) \cong \\
&\bigoplus_{n_1+n_2 = r}((\overline{\soc}_{\g'}^{(n_1+1)} V_{\lambda, 0}) \otimes (\overline{\soc}_{\g'}^{(n_2+1)} V_{0,\mu})) \cong \\
&\bigoplus_{n_1+n_2=r}\bigoplus_{\lambda'', \mu''} \bigoplus_{\substack{|\lambda'| = |\lambda''| - n_1 \\ |\mu'| = |\mu''| - n_2}} m^{a}_{\lambda, \lambda''}m^b_{\lambda'', \lambda'}m^{c}_{\mu, \mu''} m^d_{\mu'', \mu'} V'_{\lambda', 0} \otimes V'_{0,\mu'},
\end{aligned}
\end{equation}
where the second isomorphism is a corollary of Theorems \ref{thmVlambda} and \ref{thmVmu}. 

For shortness, let us denote
\begin{align*}
S^r_{\lambda, \mu} = \sum_{n_1+n_2 = r} ({\soc}_{\g'}^{(n_1+1)} V_{\lambda, 0}) \otimes ({\soc}_{\g'}^{(n_2+1)} V_{0,\mu}).
\end{align*}
Then (\ref{eqVlambdaVmu}) and Theorem 2.3 from \cite{PSt} imply
\begin{align*}
\soc_{\g'} (S^r_{\lambda, \mu} / S^{r-1}_{\lambda, \mu}) \cong
&\bigoplus_{n_1+n_2=r}\bigoplus_{\lambda'', \mu''} \bigoplus_{\substack{|\lambda'| = |\lambda''| - n_1 \\ |\mu'| = |\mu''| - n_2}} m^{a}_{\lambda, \lambda''}m^b_{\lambda'', \lambda'}m^{c}_{\mu, \mu''} m^d_{\mu'', \mu'} V'_{\lambda', \mu'}.
\end{align*}

Clearly, $\overline{\soc}_{\g'}^{(r+1)} V_{\lambda, \mu} \subseteq \soc_{\g'} (S^r_{\lambda, \mu} / S^{r-1}_{\lambda, \mu})$.
To prove the opposite inclusion let $[u] \in S^r_{\lambda, \mu} / S^{r-1}_{\lambda, \mu}$ and $[u] \neq 0$. Let without loss of generality $u = c_{\lambda}(u_1)\otimes c_{\mu}(u_2)$, where $c_{\lambda}, c_{\mu}$ denote the Young symmetrizers corresponding to the partitions $\lambda, \mu$ (see (4.2) in \cite{FH}) and 
\begin{align*}
u_1 = x_{i_1} \otimes \dots \otimes x_{i_{n_1}} \otimes  u_1', \quad u_2 = y_{j_1} \otimes \dots \otimes y_{j_{n_2}} \otimes u_2'
\end{align*}
such that $n_1 + n_2 =r$ and $u_1' \in V'^{\otimes (p - n_1)}$, $u_2' \in V_*'^{\otimes (q-n_2)}$. Let $g \in U(\g')$ be such that $g \cdot (u_1'\otimes u_2') \in V'^{\{p-n_1,q-n_2\}}$. Then 
$$g \cdot [u] \in (S^r_{\lambda, \mu} \cap V^{\{p,q\}} + S^{r-1}_{\lambda, \mu}) /S^{r-1}_{\lambda, \mu}.
$$
Therefore,
\begin{align*}
& \soc_{\g'} (S^{r}_{\lambda, \mu} / S^{r-1}_{\lambda, \mu}) \subseteq
(S^{r}_{\lambda, \mu} \cap V^{\{p,q\}} + S^{r-1}_{\lambda, \mu}) / S^{r-1}_{\lambda, \mu} \cong \overline{\soc}_{\g'}^{(r+1)} V_{\lambda, \mu}.
\end{align*} 
Thus, $\overline{\soc}_{\g'}^{(r+1)} V_{\lambda, \mu} = \soc_{\g'} (S^r_{\lambda, \mu} / S^{r-1}_{\lambda, \mu})$. This proves the statement.
\end{proof}

\subsection{Embeddings of type II} \label{secGLEmb2}

Let $\g'\subset \g$ be an embedding of type II, i.e. such that
\begin{align*}
V \cong V' \oplus N_a, \quad V_* \cong V'_* \oplus N_c. 
\end{align*}
 
By definition (see \cite{PSt}), for any partitions $\lambda$ and $\mu$ with $|\lambda| = p$ and $|\mu| = q$ we have $V_{\lambda, \mu} = V^{\{p,q\}} \cap (V_{\lambda, 0} \otimes V_{0, \mu})$. Below we determine the socle filtrations of $V_{\lambda, 0} \otimes V_{0, \mu}$ and of $V^{\{p,q\}}$ and then use property (\ref{eqPropSoc2}) of socle filtrations to obtain the socle filtration for $V_{\lambda, \mu}$.

\begin{prop} \label{propEmb2VlambdaVmu} Let $\g'\subset \g$ be an embedding of type II. Then
\begin{itemize}
\item[(i)] every $V_{\lambda, 0} \subset V^{\otimes p}$ and every $V_{0,\mu} \subset V_*^{\otimes q}$ is completely reducible over $\g'$ and
\begin{align*}
&V_{\lambda, 0} \cong \bigoplus_{\lambda'} m^{a}_{\lambda, \lambda'} V'_{\lambda',0},
&V_{0, \mu} \cong \bigoplus_{\mu'} m^{c}_{\mu, \mu'} V'_{0,\mu'};
\end{align*}
\item[(ii)] for $V_{\lambda, 0} \otimes V_{0, \mu}$ we have
\begin{align*}
\overline{\soc}_{\g'}^{(r+1)} V_{\lambda, 0} \otimes V_{0, \mu} \cong \bigoplus_{\lambda', \mu'} \bigoplus_{\substack{|\lambda''| = |\lambda'|-r \\ |\mu''| = |\mu'|-r}} \bigoplus_{\gamma} m^{a}_{\lambda, \lambda'}m^{c}_{\mu, \mu'} c^{\lambda'}_{\lambda'', \gamma}c^{\mu'}_{\mu'', \gamma} V'_{\lambda'', \mu''}.
\end{align*}
\end{itemize}
\end{prop}

\begin{proof}
Part (i). Since $V$ is semisimple over $\g'$, then so is $V^{\otimes p}$ and similarly for $V_*^{\otimes q}$. Therefore, every $V_{\lambda, 0} \subset V^{\otimes p}$ and every $V_{0,\mu} \subset V_*^{\otimes q}$ is semisimple as well. To obtain the exact multiplicities, we proceed as in the proof of Lemma \ref{lemma1_embed}.

Part (ii). From part (i) we have
\begin{align*}
V_{\lambda, 0} \otimes  V_{0, \mu} \cong \bigoplus_{\lambda', \mu'}  m^{a}_{\lambda, \lambda'} m^{c}_{\mu, \mu'} V'_{\lambda',0} \otimes V'_{0,\mu'}.
\end{align*}
Hence, property (\ref{eqPropSoc}) of socle filtrations implies
\begin{align*}
\soc^{(r)}_{\g'} (V_{\lambda, 0} \otimes  V_{0, \mu}) \cong \bigoplus_{\lambda', \mu'}  m^{a}_{\lambda, \lambda'} m^{c}_{\mu, \mu'} \soc^{(r)}_{\g'} (V'_{\lambda',0} \otimes V'_{0,\mu'}).
\end{align*}
Then, using Theorem 2.3 in \cite{PSt}, we obtain the desired formula.
\end{proof}

\begin{prop} \label{propEmb2Soc1}Let $\g' \subset \g$ be an embedding of type II. Then for $a \in \ZZ_{\geq 0} \sqcup \{\infty\}$ and $a > p+q-2$,
\begin{align*}
\overline{\soc}_{\g'}^{(r+1)}V^{\{p,q\}} \cong \bigoplus_{\{I_1, \dots, I_r\}}\bigoplus_{k = 0}^{p-r} \bigoplus_{l=0}^{q-r} \binom{p-r}{k} \binom{q-r}{l} N_a^{\{k,l\}}\otimes V'^{\{p-r-k,q-r-l\}},
\end{align*}
where $N_a^{\{k,l\}}$ is as in the proof of Proposition \ref{propEmb2Appr1}.
\end{prop}

\begin{proof}
Equation (\ref{eqEmb2SS}) from the proof of Proposition \ref{propEmb2Appr1} implies that
\begin{align*}
\bigoplus_{\{I_1, \dots, I_r\}} \Phi_{I_1, \dots, I_r}: \overline{\soc}_{\g'}^{(r+1)}V^{\{p,q\}} \rightarrow \bigoplus_{\{I_1, \dots, I_r\}} \soc_{\g'} V^{\{p-r, q-r\}}
\end{align*}
is an injective homomorphism of $\g'$-modules. Now we will show that if $a > p+q -2$, then for each $r$ the above homomoprhism is also surjective. 
Without loss of generality fix the following collection of index pairs $I_1 = (1,q-r+1), \dots, I_r = (r, q)$. 
Let $v$ be an indecomposable element of $\soc_{\g'} V^{\{p-r, q-r\}}$. By indecomposable here we mean that $v$ cannot be decomposed as a sum $v = v' + v''$ such that all monomials in $v'$ and $v''$ belong to $v$ and each of $v'$ and $v''$ is also an element of $\soc_{\g'} V^{\{p-r, q-r\}}$. Then $v$ contains at most $p+q -2r$ entries with distinct indices from the pair of dual bases $\{z_i\}$, $\{t_i\}$. Let $i_1, \dots, i_r$ be indices from $I_a$ such that neither $z_{i_k}$ nor $t_{i_k}$ for $k=1, \dots, r$ enters the expression of $v$. These exist thanks to the condition $a > p+q-2$. In addition, let $v'_k$ and $v'^*_k$ be a pair of dual elements respectively from $V'$ and $V_*'$ which do not enter the expression of $v$. Then the vector
$$
u = z_{i_1}\otimes \dots \otimes z_{i_r}\otimes v\otimes t_{i_1} \otimes \dots \otimes t_{i_r} - u_1
$$
where 
\begin{align*}
u_1 = &v'_k\otimes z_{i_2}\otimes \dots \otimes z_{i_r}\otimes v\otimes v'^*_k\otimes t_{i_2}\otimes \dots\otimes t_{i_r} + \\
&z_{i_1}\otimes v'_k\otimes z_{i_3} \otimes \dots \otimes z_{i_r}\otimes v\otimes t_{i_1}\otimes v'^*_k \otimes t_{i_3} \otimes \dots\otimes t_{i_r} + \dots + \\
&z_{i_1}\otimes \dots \otimes z_{i_{r-1}}\otimes v'_k\otimes v\otimes t_{i_1}\otimes \dots\otimes t_{i_{r-1}} \otimes v'^*_k
\end{align*}
belongs to $\soc_{\g'}^{(r+1)}V^{\{p,q\}}$  and 
$$\bigoplus_{\{I_1, \dots, I_r\}}\Phi_{I_1, \dots, I_r} (u) = v.
$$
This proves surjectivity. Then the statement follows from equation (\ref{eqEmb2Soc}) in the proof of Proposition \ref{propEmb2Appr1}.
\end{proof}

In order to obtain an exact expression for the layers of the socle filtraiton of $V^{\{p,q\}}$ also in the cases when $a \leq p+q-2$ we will use another approach. This approach covers all cases in which $a$ is finite. Therefore, in what follows, we fix $a \in \ZZ_{\geq 0}$. As is done in \cite{PSt}, for any index pair $I=(i,j)$ as above we define the inclusion
\begin{align*}
\Psi^{a}_I : N_a^{\{p-1, q-1\}} \rightarrow N_a^{\otimes (p,q)}
\end{align*}
given by
\begin{align*}
&u_1\otimes \dots \otimes u_{p-1} \otimes u_1^* \otimes \dots \otimes u_{q-1}^* \mapsto
\sum_{k=1}^{a} \dots u_{i-1}\otimes z_k\otimes u_{i+1}\otimes \dots \otimes u^*_{j-1}\otimes t_k\otimes u^*_{j+1}\otimes \dots,
\end{align*}
where $u_i$ is an arbitrary element in $N_a$ and $u^*_j$ is an arbitrary element in $N_c$. Similarly, for any disjoint collection of index pairs $I_1, \dots, I_r$, where $r = 1, \dots, \min{(p,q)}$, we define the inclusion 
\begin{align*}
\Psi^{a}_{I_1, \dots, I_r} : N_a^{\{p-r, q-r\}} \rightarrow N_a^{\otimes (p,q)}
\end{align*}
as the sum of the $r$-fold insertions of all possible ordered collections of $r$ terms of the form $z_i\otimes t_i$, including collections with repeating terms. Then, following \cite{PSt}, we denote
\begin{align*}
(N_a)^{\{p,q\}}_r = \sum_{\{I_1,\dots, I_r\}} \im \Psi^{a}_{I_1, \dots, I_r}.
\end{align*}
Furthermore, we set $(N_a)_0^{\{k,l\}} = N_a^{\{k,l\}}$. 
It is stated in \cite{PSt} that for all $a$ we have the direct sum decomposition
\begin{align*}
N_a^{\otimes (p,q)}= N_a^{\{p,q\}} \oplus (N_a)^{\{p,q\}}_1 \oplus \dots \oplus (N_a)^{\{p,q\}}_l,
\end{align*}
where $l = \min{(p,q)}$.

\begin{prop} \label{propEmb2Appr2} Let $\g' \subset \g$ be an embedding of type II such that $a = \dim N_a \in \ZZ_{\geq 0}$. Then
\begin{align*}
\overline{\soc}^{(r+1)}_{\g'}  V^{\{p,q\}} \cong \bigoplus_{k=0}^p \bigoplus_{l=0}^q \binom{p}{k} \binom{q}{l} (N_a)^{\{k,l\}}_r \otimes V'^{\{p-k, q-l\}}.
\end{align*}
\end{prop}

\begin{proof}
Note that every element $u \in \bigcap_{I_1, \dots, I_{r+1}} \ker \Phi_{I_1, \dots, I_{r+1}}$ can be written as $u = u_1 + u_2$, where 
\begin{align*}
&u_1 \in \bigoplus_{k=0}^p \bigoplus_{l=0}^q \binom{p}{k} \binom{q}{l} (N_a)^{\{k,l\}}_r \otimes V'^{\{p-k, q-l\}},
\end{align*}
\begin{align*}
&u_2 \in \bigcap_{I_1, \dots, I_{r}} \ker \Phi_{I_1, \dots, I_{r}}.
\end{align*}
Thus the statement follows from Proposition \ref{propEmb2Appr1}.
\end{proof}

Next, we check that Proposition \ref{propEmb2Soc1} and Proposition \ref{propEmb2Appr2} give the same formulas for finite dimension $a$ with $a > p +q -2$. For $r=0$ this is obvious, so we prove it for $r \geq 1$.
Note that for $k$ or $l$ smaller than $r$ we have $(N_a)^{\{k,l\}}_r = 0$ and so the formula in Proposition \ref{propEmb2Appr2} can be rewritten as
\begin{align*}
\overline{\soc}^{(r+1)}_{\g'}  V^{\{p,q\}} \cong \bigoplus_{k=0}^{p-r} \bigoplus_{l=0}^{q-r} \binom{p}{k+r} \binom{q}{l+r} (N_a)^{\{k+r,l+r\}}_r \otimes V'^{\{p-k-r, q-l-r\}}.
\end{align*}
Hence, we have to show that
\begin{align} \label{identity}
\binom{p}{k+r}\binom{q}{l+r} (N_a)^{\{k+r,l+r\}}_r = \bigoplus_{\{I_1, \dots, I_r\}}\binom{p-r}{k}\binom{q-r}{l} N_a^{\{k,l\}},
\end{align}
where $i \in \{i,\dots, p\}$ and $j \in \{1, \dots, q \}$. Having in mind that the set of all such collections of $r$ disjoint index pairs $I_1, \dots, I_r$ has $\binom{p}{r}\binom{q}{r}r!$ elements, we can further rewrite (\ref{identity}) as
\begin{align} \label{eqIdent}
(N_a)^{\{k+r,l+r\}}_r = \bigoplus_{\{J_1, \dots, J_r\}} N_a^{\{k,l\}},
\end{align} 
where the sum runs over all collections of disjoint index pairs $J_1, \dots, J_r$ with $i \in \{1, \dots, k+r\}$ and $j \in \{1, \dots, l+r\}$.
Formula (\ref{eqIdent}) holds when 
$$\sum_{\{I_1,\dots, I_r\}} \im \Psi^{a}_{I_1, \dots, I_r} = \bigoplus_{\{I_1,\dots, I_r\}} \im \Psi^{a}_{I_1, \dots, I_r}$$ 
which is true precisely when $a > k+l$.
And since $k + l \leq p+q-2 < a$ for all $r \geq 1$, this completes the proof.

Propositions \ref{propEmb2Soc1} and \ref{propEmb2Appr2} imply that in order to determine the layers of the socle filtration of $V^{\{p,q\}}$ we need to determine the dimensions of the trivial $\g'$-modules $N_a^{\{p,q\}}$ for all $a \in \ZZ_{\geq 0} \sqcup \{\infty\}$, and the dimensions of the modules $(N_a)^{\{p,q\}}_r$ for $a \in \ZZ_{\geq 0}$. Notice that $N_a^{\{p,q\}}$ has the same dimension as the $\gl(a)$-module $V_a^{\{p,q\}}$. In particular, if $a = \infty$, then $V_a^{\{p,q\}}$ is just the $\gl(\infty)$-module $V^{\{p,q\}}$, which is obviously infinite dimensional. Similarly, $(N_a)^{\{p,q\}}_r$ has the same dimension as the $\gl(a)$-module $(V_a)_r^{\{p,q\}}$ (see \cite{PSt} for the notation). Thus, it is enough to determine the dimensions of the modules $V_a^{\{p,q\}}$ and $(V_a)_r^{\{p,q\}}$ for any finite $a$.

Schur-Weyl duality (see, e.g. \cite{FH}) yields
\begin{align} \label{eqSchur}
V_a^{\otimes (p,q)} \cong \bigoplus_{\substack{|\lambda| = p \\ |\mu| = q}} V^{a}_{\lambda, 0} \otimes V^{a}_{0,\mu} \otimes (H_{\lambda}\otimes H_{\mu}).
\end{align}
Here, $H_{\lambda}$ (resp., $H_{\mu}$) denotes the irreducible representation of the symmetric group $\SS_{p}$ (resp., $\SS_q$) corresponding to the partition $\lambda$ (resp., $\mu$). $V^{a}_{\lambda, 0}$ denotes as before the irreducible $\gl(a)$-module with highest weight $(\lambda, 0)$.

Furthermore, we have that
\begin{align*}
V^{a}_{\lambda, 0} \otimes V^{a}_{0,\mu} \cong \bigoplus_{\lambda', \mu'} \tilde{c}^{\lambda, \mu}_{\lambda', \mu'} V^{a}_{\lambda', \mu'},
\end{align*}
where
\begin{align} \label{eqLRCoeff}
\tilde{c}^{\lambda, \mu}_{\lambda', \mu'} = \bigoplus_{\gamma} c^{\lambda}_{\lambda', \gamma}c^{\mu}_{\mu', \gamma}
\end{align}
for $a > p+q$ (see e.g. \cite{K}, \cite{HTW}). For $a \leq p+q$, $\tilde{c}^{\lambda, \mu}_{\lambda', \mu'}$ can be obtained from (\ref{eqLRCoeff}) by the modification rules 
in \cite{K}. Then
\begin{align*}
V_a^{\otimes (p,q)} \cong \bigoplus_{\substack{|\lambda| = p \\ |\mu| = q}} \bigoplus_{\lambda', \mu'} \tilde{c}^{\lambda, \mu}_{\lambda', \mu'} V^{a}_{\lambda', \mu'} \otimes (H_{\lambda}\otimes H_{\mu}).
\end{align*}

It is stated in \cite{PSt} that
$
V_a^{\otimes (p,q)} = (V_a)_0^{\{p,q\}} \oplus (V_a)_1^{\{p,q\}} \oplus \dots \oplus (V_a)_l^{\{p,q\}}.
$
Furthermore, if $\Phi^{a}_{I_1, \dots,I_{r}}$ denotes the $r$-fold contraction with respect to the bilinear pairing on $V_a \times V^*_a$ then
\begin{align*}
\bigcap_{I_1, \dots, I_{r+1}} \ker \Phi^{a}_{I_1, \dots,I_{r+1}} = (V_a)_0^{\{p,q\}} \oplus (V_a)_1^{\{p,q\}} \oplus \dots \oplus (V_a)_r^{\{p,q\}}
\end{align*}
for $r = 0, \dots, l$.

Next, we prove that each $(V_a)_r^{\{p,q\}}$ consists of simple modules $V^{a}_{\lambda, \mu}$ with $|\lambda| = p-r$ and $|\mu| = q-r$. On the one hand, (\ref{eqSchur}) implies that a simple module $V^{a}_{\lambda, \mu}$ with $|\lambda| = p-r$ and $|\mu| = q-r$ cannot be realized as a submodule of $V_a^{\otimes (p-r-1, q-r-1)}$. Hence, $V^{a}_{\lambda, \mu} \subset \bigcap_{I_1, \dots, I_{r+1}} \ker \Phi^{a}_{I_1, \dots,I_{r+1}}$. 
On the other hand, each contraction $\Phi^a_{I_1, \dots, I_r} : V_a^{\otimes (p,q)} \rightarrow V_a^{\otimes (p-r, q-r)}$ is surjective, and hence at least one copy of the module $V^{a}_{\lambda, \mu}$ inside $V_a^{\otimes (p,q)}$ belongs to the image of $\Phi^a_{I_1, \dots, I_r}$. But all copies of $V^{a}_{\lambda, \mu}$ inside $V_a^{\otimes (p,q)}$ are isomorphic, hence every $V^{a}_{\lambda, \mu}$ inside $V_a^{\otimes (p,q)}$ belongs to the image of a contraction $\Phi^{a}_{I_1, \dots,I_{r}}$. This implies that $V^{a}_{\lambda, \mu} \notin \bigcap_{I_1, \dots, I_{r}} \ker \Phi^{a}_{I_1, \dots,I_{r}}$.
Therefore,
\begin{align*}
(V_a)_r^{\{p,q\}} \cong \bigoplus_{\substack{|\lambda| = p \\ |\mu| = q}} \bigoplus_{\substack{|\lambda'|= p-r\\ |\mu'| = q-r}} \tilde{c}^{\lambda, \mu}_{\lambda', \mu'} V^{a}_{\lambda', \mu'} \otimes (H_{\lambda}\otimes H_{\mu}).
\end{align*}

Finally, if we denote $K^{(r+1)}_{p,q} = \dim (V_a)_r^{\{p,q\}}$ then
\begin{align*}
K^{(r+1)}_{p,q} = \sum_{\substack{|\lambda| = p \\ |\mu| = q}} \sum_{\substack{|\lambda'|= p-r\\ |\mu'| = q-r}} \tilde{c}^{\lambda, \mu}_{\lambda', \mu'} \dim V^{a}_{\lambda', \mu'} \dim H_{\lambda} \dim H_{\mu}.
\end{align*}
In particular, for all $a \in \ZZ_{\geq 0} \sqcup \{\infty\}$
\begin{align*}
K^{(1)}_{p,q} = \sum_{\substack{|\lambda| = p \\ |\mu| = q}} \dim V^{a}_{\lambda, \mu} \dim H_{\lambda} \dim H_{\mu}.
\end{align*}

The following theorem now follows from Propositions \ref{propEmb2VlambdaVmu}, \ref{propEmb2Soc1}, and \ref{propEmb2Appr2} and property (\ref{eqPropSoc2}) of socle filtrations. 
\begin{thm} \label{thmEmb2} Let $\g'\subset \g$ be an embedding of type II, and let $V_{\lambda, \mu} \subset V^{\{p,q\}}$. Then
\begin{align*}
\overline{\soc}_{\g'}^{(r+1)} V_{\lambda, \mu} \cong \bigoplus_{k=0}^{p-r}\bigoplus_{l=0}^{q-r} \bigoplus_{\substack{|\lambda'| = p-k \\ |\mu'| = q-l}}\bigoplus_{\substack{|\lambda''| = |\lambda'|-r \\ |\mu''| = |\mu'| - r}} T^{\lambda, \mu}_{\lambda', \mu',\lambda'', \mu''} V'_{\lambda'', \mu''},
\end{align*}
where, for $a \in \ZZ_{\geq 0}$
\begin{align*}
T^{\lambda, \mu}_{\lambda', \mu',\lambda'', \mu''} = \min (&\sum_{\gamma} m^{a}_{\lambda, \lambda'} m^{c}_{\mu, \mu'} c^{\lambda'}_{\lambda'', \gamma}c^{\mu'}_{\mu'', \gamma} ,
\binom{p}{k+r}\binom{q}{l+r} K^{(r+1)}_{k+r,l+r} \dim H_{\lambda''} \dim H_{\mu''}),
\end{align*}
and for $a \in \ZZ_{\geq 0} \sqcup \{\infty\}$ with $ a> p+q -2$
\begin{align*}
T^{\lambda, \mu}_{\lambda', \mu',\lambda'', \mu''} = \min (&\sum_{\gamma} m^{a}_{\lambda, \lambda'} m^{c}_{\mu, \mu'} c^{\lambda'}_{\lambda'', \gamma}c^{\mu'}_{\mu'', \gamma} ,
r!\binom{p}{r}\binom{q}{r} \binom{p-r}{k}\binom{q-r}{l} K^{(1)}_{k,l} \dim H_{\lambda''} \dim H_{\mu''}).
\end{align*}
\end{thm}

The above discussion 
leads to the following combinatorial identity, which connects the dimensions of certain simple representations of the symmetric group with Littlewood-Richardson coefficients.

\begin{prop} For any two partitions $\lambda'$ and $\mu'$ with $|\lambda'| = p-r$, $|\mu'| = q-r$ for some integers $p,q,r$ the following holds:
\begin{align*}
\binom{p}{r}\binom{q}{r} r! \dim H_{\lambda'} \dim H_{\mu'} = \sum_{\substack{|\lambda| = p \\ |\mu| = q}} \sum_{\gamma} c^{\lambda}_{\lambda', \gamma} c^{\mu}_{\mu', \gamma} \dim H_{\lambda} \dim H_{\mu}.
\end{align*}
\end{prop}

\begin{proof}
For $a > p+q-2$ the following map is an isomorphism:
$$
\bigoplus_{\{I_1, \dots, I_r\}} \Phi^a_{I_1, \dots, I_r} : (V_a)_r^{\{p,q\}} \rightarrow \bigoplus_{\{I_1, \dots, I_r\}} V_a^{\{p-r, q-r\}}.
$$
This proves the statement.
\end{proof}

\subsection{Embeddings of type III}
In this section we consider embeddings $\g' \subset \g$ of type III, i.e. for which
$$V \cong kV'\oplus lV'_*, \quad V_* \cong lV' \oplus kV'_*.$$

First we derive the branching rule for diagonal embeddings $\gl(n) \subset \gl(kn + ln)$ of signature $(k,l,0)$. 
We decompose the embedding $\gl(n) \subset \gl(kn + ln)$ in the following standard way:
\begin{align} \label{eqEmbSplit}
\gl(n) \subset \gl(n) \oplus \gl(n) \cong \gl(n) \oplus \gl(n) \subset \gl(kn) \oplus \gl(ln) \subset \gl(kn + ln),
\end{align}
where the isomorphism $\gl(n) \oplus \gl(n) \cong \gl(n) \oplus \gl(n)$ is given by 
$$
(A, B) \mapsto (A, -B^T)
$$
for any $A, B \in  \gl(n)$. The other maps in (\ref{eqEmbSplit}) are the obvious ones. Recall from Proposition \ref{propBranRuleDiag} from Section \ref{secFinBran} that for diagonal embeddings $\gl(n) \subset \gl(kn)$ the following branching rule holds:
\begin{align} \label{eqDiagChap2}
{V^{kn}_{\lambda, \mu}}_{\downarrow \gl(n)} \cong  \bigoplus_{\substack{\beta^+_1, \dots, \beta^+_k \\ \beta^-_1, \dots, \beta^-_k \\ \lambda', \mu'}}
C^{(\lambda, \mu)}_{(\beta^+_1, \dots, \beta^+_k) (\beta^-_1, \dots, \beta^-_k)} D^{(\lambda', \mu')}_{(\beta^+_1, \dots, \beta^+_k) (\beta^-_1, \dots, \beta^-_k)} V^n_{\lambda', \mu'},
\end{align}
where the coefficients $C^{(\lambda, \mu)}_{(\beta^+_1, \dots, \beta^+_k) (\beta^-_1, \dots, \beta^-_k)}$ and $D^{(\lambda', \mu')}_{(\beta^+_1, \dots, \beta^+_k) (\beta^-_1, \dots, \beta^-_k)}$ are as in Section \ref{secFinBran}.

Thus, the decomposition (\ref{eqEmbSplit}), formulas 2.1.1 and 2.2.1 from \cite{HTW}, and equation (\ref{eqDiagChap2}) yield
\begin{equation} \label{eqDiagBranch}
\begin{aligned}
&{V^{(k+l)n}_{\lambda, \mu}}_{\downarrow \gl(n)} \cong &
\bigoplus c^{(\lambda, \mu)}_{(\gamma^+, \gamma^-), (\delta^+, \delta^-)} C^{(\gamma^+, \gamma^-)}_{(\alpha^+_1, \dots, \alpha^+_k) (\alpha^-_1, \dots, \alpha^-_k)} D^{(\sigma^+, \sigma^-)}_{(\alpha^+_1, \dots, \alpha^+_k) (\alpha^-_1, \dots, \alpha^-_k)} \\
&&C^{(\delta^+, \delta^-)}_{(\beta^+_1, \dots, \beta^+_l) (\beta^-_1, \dots, \beta^-_l)} D^{(\tau^+, \tau^-)}_{(\beta^+_1, \dots, \beta^+_l) (\beta^-_1, \dots, \beta^-_l)} d^{(\lambda', \mu')}_{(\sigma^+, \sigma^-), (\tau^-, \tau^+)} V^n_{\lambda', \mu'},
\end{aligned}
\end{equation}
where the sum is over all partitions $\gamma^+, \gamma^-, \delta^+, \delta^-$, $\alpha^+_1, \dots, \alpha^+_k$, $\alpha^-_1, \dots, \alpha^-_k$, $\sigma^+, \sigma^-$, $\beta^+_1, \dots, \beta^+_l$, $\beta^-_1, \dots, \beta^-_l$, $\tau^+, \tau^-$, $\lambda', \mu'$. 

\begin{thm} \label{thmEmb3}
Consider an embedding $\g' \subset \g$ of type III. Then for each $V_{\lambda,\mu} \subset V^{\otimes (p,q)}$
\begin{align*}
\overline{\soc}_{\g'}^{(r+1)}V_{\lambda, \mu} \cong \bigoplus_{m=0}^p \bigoplus_{n=0}^q \bigoplus_{\substack{|\lambda'| = m+n-r \\|\mu'| = p+q-m-n-r}} A^{\lambda, \mu}_{\lambda', \mu'}V'_{\lambda', \mu'},
\end{align*}
where 
\begin{align*}
A^{\lambda, \mu}_{\lambda', \mu'} = &\sum c^{(\lambda, \mu)}_{(\gamma^+, \gamma^-), (\delta^+, \delta^-)} C^{(\gamma^+, \gamma^-)}_{(\alpha^+_1, \dots, \alpha^+_k) (\alpha^-_1, \dots, \alpha^-_k)} D^{(\sigma^+, \sigma^-)}_{(\alpha^+_1, \dots, \alpha^+_k) (\alpha^-_1, \dots, \alpha^-_k)} \\
&C^{(\delta^+, \delta^-)}_{(\beta^+_1, \dots, \beta^+_l) (\beta^-_1, \dots, \beta^-_l)} D^{(\tau^+, \tau^-)}_{(\beta^+_1, \dots, \beta^+_l) (\beta^-_1, \dots, \beta^-_l)} d^{(\lambda', \mu')}_{(\sigma^+, \sigma^-), (\tau^-, \tau^+)}.
\end{align*}
\end{thm}

\begin{proof}
Property (\ref{eqPropSoc}) of socle filtrations implies
\begin{align} \label{eqSocVpq}
\soc_{\g'}^{(r+1)} V^{\otimes (p, q)} \cong \bigoplus_{m=0}^p \bigoplus_{n=0}^q \binom{p}{m} \binom{q} {n} k^{(m+q-n)} l^{(n+p-m)} \soc_{\g'}^{(r+1)} V'^{\otimes (m+n, p+q-m-n)}.
\end{align}
Then from (\ref{eqSocVpq}) and from Theorem 2.2 in \cite{PSt} we obtain
\begin{align*}
\overline{\soc}_{\g'}^{(r+1)}V_{\lambda, \mu} \cong \bigoplus_{m=0}^p \bigoplus_{n=0}^q \bigoplus_{\substack{|\lambda'| = m+n-r \\|\mu'| = p+q-m-n-r}} a^{\lambda, \mu}_{\lambda', \mu'}V'_{\lambda', \mu'}
\end{align*}
for some multiplicities $a^{\lambda, \mu}_{\lambda', \mu'}$. Since on the different layers of the socle filtration non-isomorphic modules appear we can compute the multiplicities $a^{\lambda, \mu}_{\lambda', \mu'}$ in the following way. We take bases of $V', V'_*, V$, and $V_*$ as in the settings before Proposition \ref{propBasis}.
Then we construct exhaustions of $\g'$ and $\g$ to obtain a commutative diagram of embeddings in which all vertical arrows have signature $(k,l,0)$. Thus from (\ref{eqDiagBranch}) we obtain the values of the multiplicities.
\end{proof}

For all other cases of embeddings the derivation of the branching laws follows the same ideas as for the pair $\g', \g \cong \gl(\infty)$. Therefore, we do not give here the explicit computations and list only the end results in the Appendix.

\newpage
\section{Appendix}
\FloatBarrier
\begin{table}[h]
\centering
\begin{tabular} {|c|l|}
\hline
\multicolumn{2}{|c|} {Embeddings of type I} \\
\hline
& \\
{$\gl(V', V'_*) \subset \gl(V,V_*)$} & \(\displaystyle \soc_{\g'}^{(r+1)} V_{\lambda, \mu} =(\sum_{n_1 + n_2 = r} \soc^{(n_1+1)}_{\g'} V_{\lambda,0}\otimes \soc^{(n_2+1)}_{\g'} V_{0, \mu}) \cap V^{\{p, q \}} \) \\
$\soc_{\g'}V \cong V' \oplus N_a$, $V/\soc_{\g'} V \cong N_b,$ &\\
$\soc_{\g'}V_* \cong V' \oplus N_c$, $V_*/\soc_{\g'} V_* \cong N_d,$&\\
$\left\langle N_a, N_c \right\rangle = 0$ & \( \displaystyle
 \overline{\soc}_{\g'}^{(r+1)} V_{\lambda, \mu} \cong \bigoplus_{\substack{n_1 + n_2 = r \\ \lambda'', \mu''}} \bigoplus_{\substack{|\lambda'| = |\lambda''| - n_1 \\ |\mu'| = |\mu''| - n_2}}m^{a}_{\lambda, \lambda''} m^{b}_{\lambda'', \lambda'} m^{c}_{\mu, \mu''} m^{d}_{\mu'', \mu'} V'_{\lambda', \mu'} \)\\
 &\\
&(see Section \ref{secGLEmb1} for the notations)\\
\hline
& \\
{$\sp(V') \subset \sp(V)$} & \(\displaystyle \soc_{\g'}^{(r+1)} V_{\left\langle \lambda \right\rangle} = \soc^{(r+1)}_{\gl(V',V')} V_{\lambda,0} \cap V^{\left\langle p \right\rangle} \) \\
&\\
$\soc_{\g'}V \cong V' \oplus N_a$, $V/\soc_{\g'} V \cong N_b,$&\\
$\left\langle N_a, N_a \right\rangle = 0$ & \( \displaystyle
 \overline{\soc}_{\g'}^{(r+1)} V_{\left\langle \lambda \right\rangle} \cong \bigoplus_{\lambda''} \bigoplus_{|\lambda'| = |\lambda''| - r}m^{a}_{\lambda, \lambda''} m^{b}_{\lambda'', \lambda'} V'_{\left\langle \lambda' \right\rangle} \)\\
\hline
& \\
$\so(V') \subset \so(V)$ & \(\displaystyle \soc_{\g'}^{(r+1)} V_{[\lambda]} = \soc^{(r+1)}_{\gl(V',V')} V_{\lambda,0} \cap V^{[p]} \) \\
&\\
$\soc_{\g'}V \cong V' \oplus N_a$, $V/\soc_{\g'} V \cong N_b,$&\\
$\left\langle N_a, N_a \right\rangle = 0$ & \(\displaystyle
\overline{\soc}_{\g'}^{(r+1)} V_{[\lambda]} \cong \bigoplus_{\lambda''} \bigoplus_{|\lambda'| = |\lambda''| - r}m^{a}_{\lambda, \lambda''} m^{b}_{\lambda'', \lambda'} V'_{[ \lambda']} \) \\
\hline

\end{tabular}
\end{table}

\FloatBarrier

\FloatBarrier
\begin{table}[h]
\centering
\begin{tabular} {|c|l|}
\hline
\multicolumn{2}{|c|}{Embeddings of type II} \\
\hline
&\\
{$\gl(V', V'_*) \subset \gl(V, V_*)$} & \( \displaystyle
\overline{\soc}_{\g'}^{(r+1)} V_{\lambda, \mu} \cong \bigoplus_{k=0}^{p-r}\bigoplus_{l=0}^{q-r} \bigoplus_{\substack{|\lambda'| = p-k \\ |\mu'| = q-l}}\bigoplus_{\substack{|\lambda''| = |\lambda'|-r \\ |\mu''| = |\mu'| - r}} T^{\lambda, \mu}_{\lambda', \mu',\lambda'', \mu''} V'_{\lambda'', \mu''}, \)\\
$V \cong V' \oplus N_a$& where for $a \in \ZZ_{\geq 0}$ \\
$V_* \cong V'_* \oplus N_c$ &\\
&\(\displaystyle T^{\lambda, \mu}_{\lambda', \mu',\lambda'', \mu''} = \min (\sum_{\gamma} m^{a}_{\lambda, \lambda'} m^{c}_{\mu, \mu'} c^{\lambda'}_{\lambda'', \gamma}c^{\mu'}_{\mu'', \gamma} ,\)\\
&\(\displaystyle 
\binom{p}{k+r}\binom{q}{l+r} K^{(r+1)}_{k+r,l+r} \dim H_{\lambda''} \dim H_{\mu''}) \)\\
&\\
&and for $a \in \ZZ_{\geq 0} \sqcup \{\infty\}$ with $ a> p+q -2$ \\
&\\
&\(\displaystyle T^{\lambda, \mu}_{\lambda', \mu',\lambda'', \mu''}  = \min (\sum_{\gamma} m^{a}_{\lambda, \lambda'} m^{c}_{\mu, \mu'} c^{\lambda'}_{\lambda'', \gamma}c^{\mu'}_{\mu'', \gamma} \) ,\\
&\(\displaystyle 
r!\binom{p}{r}\binom{q}{r} \binom{p-r}{k}\binom{q-r}{l} K^{(1)}_{k,l} \dim H_{\lambda''} \dim H_{\mu''}). \) \\
&\\
& \(\displaystyle K^{(r+1)}_{p,q} = \sum_{\substack{|\lambda| = p \\ |\mu| = q}} \sum_{\substack{|\lambda'|= p-r\\ |\mu'| = q-r}} \tilde{c}^{\lambda, \mu}_{\lambda', \mu'} \dim V^{a}_{\lambda', \mu'} \dim H_{\lambda} \dim H_{\mu}, \text{ for } a \in \ZZ_{\geq 0}\) \\
& \(\displaystyle K^{(1)}_{p,q} = \sum_{\substack{|\lambda| = p \\ |\mu| = q}} \dim V^{a}_{\lambda, \mu} \dim H_{\lambda} \dim H_{\mu}, \text{ for } a \in \ZZ_{\geq 0} \sqcup \{\infty\} \) \\
&(see Section \ref{secGLEmb2} for the notations)\\
\hline

\end{tabular}
\end{table}
\FloatBarrier

\FloatBarrier
\begin{table}
\centering
\begin{tabular} {|l|l|}
\hline
\multicolumn{2}{|c|}{Embeddings of type II} \\
\hline
&\\
{$\sp(V') \subset \sp(V)$} & \( \displaystyle
\overline{\soc}_{\g'}^{(r+1)} V_{\left\langle \lambda \right\rangle} \cong \bigoplus_{s=0}^{d-2r} \bigoplus_{|\lambda'| = d-s} \bigoplus_{|\lambda''| = |\lambda'| - 2r}T^{\lambda}_{\lambda', \lambda''} V'_{\left\langle \lambda'' \right\rangle}, \)\\
& \\
$V \cong V' \oplus N_a$ & where for $a \in 2\ZZ_{\geq 0}$ \\
&\(\displaystyle T^{\lambda}_{\lambda', \lambda''} = \min \{\sum_{\gamma} m^a_{\lambda, \lambda'} c^{\lambda'}_{\lambda'' (2\gamma)^T} , \binom{d}{s+2r}K^{(r+1)}_{s+2r} \dim H_{\lambda''} \} \)\\
&\\
&and for $a \in 2\ZZ_{\geq 0} \sqcup \{\infty\}$ with $a > 2d - 2$ \\
&\(\displaystyle T^{\lambda}_{\lambda', \lambda''} = \min \{\sum_{\gamma} m^a_{\lambda, \lambda'} c^{\lambda'}_{\lambda'' (2\gamma)^T} , r!\binom{d}{r}\binom{d}{r} \binom{d-2r}{s}K^{(1)}_s \dim H_{\lambda''} \}. \) \\
& \(\displaystyle K^{(r+1)}_d = \sum_{|\lambda| = d} \sum_{|\lambda'| = d-2r} \tilde{c}^{\lambda}_{\lambda'} \dim V^a_{\left\langle \lambda' \right\rangle} \dim H_{\lambda}, \text{ for } a \in 2\ZZ_{\geq 0} \) \\
& \(\displaystyle K^{(1)}_d = \sum_{|\lambda| = d}\dim V^a_{\left\langle \lambda \right\rangle} \dim H_{\lambda}, \text{ for } a \in 2\ZZ_{\geq 0} \sqcup \{\infty\} \) \\
& $V^a_{\left\langle \lambda \right\rangle}$ denotes the simple $\sp(a)$-module with highest weight $\lambda$\\
&\\
& $H_{\lambda}$ is the simple $\SS_p$-module corresponding to $\lambda$ \\
&\\
&\(\displaystyle \tilde{c}^{\lambda}_{\lambda'}  \text{ is defined by } {V^a_{\lambda, 0}}_{\downarrow \sp(a)} \cong \bigoplus_{\lambda'} \tilde{c}^{\lambda}_{\lambda'} V^a_{\left\langle \lambda' \right\rangle} \text{  (\cite{HTW}, \cite{K}) }\) \\
\hline
&\\
{$\so(V') \subset \so(V)$} & \( \displaystyle
\overline{\soc}_{\g'}^{(r+1)} V_{[\lambda]} \cong \bigoplus_{s=0}^{d-2r} \bigoplus_{|\lambda'| = d-s} \bigoplus_{|\lambda''| = |\lambda'| - 2r}T^{\lambda}_{\lambda', \lambda''} V'_{[\lambda'']}, \)\\
& \\
$V \cong V' \oplus N_a$ & where for $a \in \ZZ_{\geq 0}$ \\
&\(\displaystyle T^{\lambda}_{\lambda', \lambda''} = \min \{\sum_{\gamma} m^a_{\lambda, \lambda'} c^{\lambda'}_{\lambda'' 2\gamma} , \binom{d}{s+2r}K^{(r+1)}_{s+2r} \dim H_{\lambda''} \} \)\\
&and for $a \in \ZZ_{\geq 0} \sqcup \{\infty\}$ with $a > 2d-2$ \\
&\(\displaystyle T^{\lambda}_{\lambda', \lambda''} = \min \{\sum_{\gamma} m^a_{\lambda, \lambda'} c^{\lambda'}_{\lambda'' 2\gamma} , r!\binom{d}{r}\binom{d}{r} \binom{d-2r}{s}K^{(1)}_s \dim H_{\lambda''} \}. \) \\
& \(\displaystyle K^{(r+1)}_d = \sum_{|\lambda| = d} \sum_{|\lambda'| = d-2r} \tilde{c}^{\lambda}_{\lambda'} \dim V^a_{[\lambda']} \dim H_{\lambda}, \text{ for } a \in \ZZ_{\geq 0} \) \\
& \(\displaystyle K^{(1)}_d = \sum_{|\lambda| = d}\dim V^a_{[\lambda]} \dim H_{\lambda}, \text{ for } a \in \ZZ_{\geq 0} \sqcup \{\infty\} \) \\
& $V^a_{[\lambda]}$ denotes the simple $\so(a)$-module with highest weight $\lambda$\\
&\\
& $H_{\lambda}$ is the simple $\SS_p$-module corresponding to $\lambda$ \\
&\\
&\(\displaystyle \tilde{c}^{\lambda}_{\lambda'}  \text{ is defined by } {V^a_{\lambda, 0}}_{\downarrow \so(a)} \cong \bigoplus_{\lambda'} \tilde{c}^{\lambda}_{\lambda'} V^a_{[\lambda']} \text{  (\cite{HTW}, \cite{K}) }\) \\
\hline
\end{tabular}
\end{table}
\FloatBarrier

\FloatBarrier
\begin{table}
\centering
\begin{tabular} {|l|l|}
\hline
\multicolumn{2}{|c|}{Embeddings of type III} \\
\hline
& \\
$\gl(V', V'_*) \subset \gl(V, V_*)$ & \(\displaystyle \overline{\soc}_{\g'}^{(r+1)}V_{\lambda, \mu} \cong \bigoplus_{m=0}^p \bigoplus_{n=0}^q \bigoplus_{\substack{|\lambda'| = m+n-r \\|\mu'| = p+q-m-n-r}} A^{\lambda, \mu}_{\lambda', \mu'}V'_{\lambda', \mu'}, \) \\
$V \cong kV' \oplus lV'_*$& where \\
$V_* \cong lV' \oplus kV'_*$ &\\
&\(\displaystyle A^{\lambda, \mu}_{\lambda', \mu'} = \sum c^{(\lambda, \mu)}_{(\gamma^+, \gamma^-), (\delta^+, \delta^-)} C^{(\gamma^+, \gamma^-)}_{(\alpha^+_1, \dots, \alpha^+_k) (\alpha^-_1, \dots, \alpha^-_k)} D^{(\sigma^+, \sigma^-)}_{(\alpha^+_1, \dots, \alpha^+_k) (\alpha^-_1, \dots, \alpha^-_k)} \)\\
&\(\displaystyle C^{(\delta^+, \delta^-)}_{(\beta^+_1, \dots, \beta^+_l) (\beta^-_1, \dots, \beta^-_l)} D^{(\tau^+, \tau^-)}_{(\beta^+_1, \dots, \beta^+_l) (\beta^-_1, \dots, \beta^-_l)} d^{(\lambda', \mu')}_{(\sigma^+, \sigma^-), (\tau^-, \tau^+)} \) \\
&\\
&(see Section \ref{secFinBran} for the notations)\\
\hline
& \\
$\sp(V') \subset \sp(V)$ & \(\displaystyle \overline{\soc}_{\g'}^{(r+1)} V_{\left\langle \lambda \right\rangle} \cong \bigoplus_{|\lambda'| = d - 2r}C^{\lambda}_{\lambda'}V'_{\left\langle \lambda' \right\rangle}, \) \\
$V \cong kV'$ & where \\
&\(\displaystyle C^{\lambda}_{\lambda'} = \sum_{\mu_1, \dots, \mu_k, \lambda'} A^{\lambda}_{\mu_1, \dots, \mu_k} B^{\lambda'}_{\mu_1, \dots, \mu_k}, \text{ for } k > 2\) \\
&\(\displaystyle C^{\lambda}_{\lambda'} = \sum_{\mu, \nu, \lambda'} a^{\lambda}_{\mu, \nu} b^{\lambda'}_{\mu, \nu}, \text{ for } k = 2; \) \\
& \(\displaystyle A^{\lambda}_{\mu_1, \dots, \mu_k} = \sum_{\alpha_1, \dots, \alpha_{k-2}} a^{\lambda}_{\alpha_1, \mu_1} a^{\alpha_1}_{\alpha_2, \mu_2} \dots a^{\alpha_{k-3}}_{\alpha_{k-2}, \mu_{k-2}} a^{\alpha_{k-2}}_{\mu_{k-1}, \mu_k}, \) \\
& \(\displaystyle B^{\lambda}_{\mu_1, \dots, \mu_k} = \sum_{\alpha_1, \dots, \alpha_{k-2}} b^{\alpha_1}_{\mu_1, \mu_2} b^{\alpha_2}_{\alpha_1, \mu_3} \dots b^{\alpha_{k-2}}_{\alpha_{k-3}, \mu_{k-1}} b^{\lambda}_{\alpha_{k-2}, \mu_k}, \) \\
& \(\displaystyle a^{\lambda}_{\mu, \nu} = \sum_{\delta, \gamma} c^{\gamma}_{\mu \nu} c^{\lambda}_{\gamma (2\delta)^T}, \)\\
&\(\displaystyle b^{\lambda}_{\mu, \nu} = \sum_{\alpha, \beta, \gamma} c^{\lambda}_{\alpha \beta}c^{\mu}_{\alpha \gamma}c^{\nu}_{\beta \gamma} \) \\
\hline
& \\
$\so(V') \subset \so(V)$ & \(\displaystyle \overline{\soc}_{\g'}^{(r+1)} V_{[\lambda]} \cong \bigoplus_{|\lambda'| = d - 2r}C^{\lambda}_{\lambda'}V'_{[\lambda']}, \) \\
$V \cong kV'$ & where \\
&\(\displaystyle C^{\lambda}_{\lambda'} = \sum_{\mu_1, \dots, \mu_k, \lambda'} A^{\lambda}_{\mu_1, \dots, \mu_k} B^{\lambda'}_{\mu_1, \dots, \mu_k}, \text{ for } k > 2\) \\
&\(\displaystyle C^{\lambda}_{\lambda'} = \sum_{\mu, \nu, \lambda'} a^{\lambda}_{\mu, \nu} b^{\lambda'}_{\mu, \nu}, \text{ for } k = 2; \) \\
& \(\displaystyle A^{\lambda}_{\mu_1, \dots, \mu_k} = \sum_{\alpha_1, \dots, \alpha_{k-2}} a^{\lambda}_{\alpha_1, \mu_1} a^{\alpha_1}_{\alpha_2, \mu_2} \dots a^{\alpha_{k-3}}_{\alpha_{k-2}, \mu_{k-2}} a^{\alpha_{k-2}}_{\mu_{k-1}, \mu_k}, \) \\
& \(\displaystyle B^{\lambda}_{\mu_1, \dots, \mu_k} = \sum_{\alpha_1, \dots, \alpha_{k-2}} b^{\alpha_1}_{\mu_1, \mu_2} b^{\alpha_2}_{\alpha_1, \mu_3} \dots b^{\alpha_{k-2}}_{\alpha_{k-3}, \mu_{k-1}} b^{\lambda}_{\alpha_{k-2}, \mu_k}, \) \\
& \(\displaystyle a^{\lambda}_{\mu, \nu} = \sum_{\delta, \gamma} c^{\gamma}_{\mu \nu} c^{\lambda}_{\gamma 2\delta}, \)\\
&\(\displaystyle b^{\lambda}_{\mu, \nu} = \sum_{\alpha, \beta, \gamma} c^{\lambda}_{\alpha \beta}c^{\mu}_{\alpha \gamma}c^{\nu}_{\beta \gamma} \) \\
\hline
\end{tabular}
\end{table}
\FloatBarrier

\FloatBarrier
\begin{table}
\centering
\begin{tabular} {|l|l|}
\hline
\multicolumn{2}{|c|}{Embeddings of type III} \\
\hline
& \\
$\sp(V') \subset \gl(V, V_*)$ & \(\displaystyle \overline{\soc}_{\g'} V_{\lambda, \mu} \cong \bigoplus_{|\sigma| = p+q - 2r} A^{\lambda, \mu}_{\sigma} V'_{\left\langle \sigma \right\rangle}, \) \\
$V \cong V_* \cong kV'$ & where for $k = 1$ \\
&\(\displaystyle A^{\lambda, \mu}_{\sigma}  = \sum_{\alpha, \beta, \gamma} c^{\sigma}_{\alpha \beta} c^{\lambda}_{\alpha (2\gamma)^T} c^{\mu}_{\beta (2\delta)^T}, \) \\
& and for $k \geq 2$ \\
& \(\displaystyle A^{\lambda, \mu}_{\sigma} = \sum_{\substack{\alpha^+_1, \dots, \alpha^+_k \\ \alpha^-_1, \dots, \alpha^-_k}} \sum_{ \substack{\lambda', \mu' \\ \alpha, \beta, \gamma}} C^{(\lambda, \mu)}_{(\alpha^+_1, \dots, \alpha^+_k) (\alpha^-_1, \dots, \alpha^-_k)} D^{(\lambda', \mu')}_{(\alpha^+_1, \dots, \alpha^+_k) (\alpha^-_1, \dots, \alpha^-_k)}c^{\sigma}_{\alpha \beta} c^{\lambda'}_{\alpha (2\gamma)^T} c^{\mu'}_{\beta (2\delta)^T}\) \\
& (see Section \ref{secFinBran} for the notations)\\
\hline
& \\
$\so(V') \subset \gl(V, V_*)$ & \(\displaystyle \overline{\soc}_{\g'} V_{\lambda, \mu} \cong \bigoplus_{|\sigma| = p+q - 2r} B^{\lambda, \mu}_{\sigma} V'_{[\sigma]}, \) \\
$V \cong V_* \cong kV'$ & where for $k = 1$ \\
&\(\displaystyle B^{\lambda, \mu}_{\sigma}  = \sum_{\alpha, \beta, \gamma} c^{\sigma}_{\alpha \beta} c^{\lambda}_{\alpha (2\gamma)} c^{\mu}_{\beta (2\delta)}, \) \\
& and for $k \geq 2$ \\
& \(\displaystyle B^{\lambda, \mu}_{\sigma} = \sum_{\substack{\alpha^+_1, \dots, \alpha^+_k \\ \alpha^-_1, \dots, \alpha^-_k}} \sum_{ \substack{\lambda', \mu' \\ \alpha, \beta, \gamma}} C^{(\lambda, \mu)}_{(\alpha^+_1, \dots, \alpha^+_k) (\alpha^-_1, \dots, \alpha^-_k)} D^{(\lambda', \mu')}_{(\alpha^+_1, \dots, \alpha^+_k) (\alpha^-_1, \dots, \alpha^-_k)}c^{\sigma}_{\alpha \beta} c^{\lambda'}_{\alpha (2\gamma)} c^{\mu'}_{\beta (2\delta)}\) \\
\hline
& \\
$\gl(V', V'_*) \subset \sp(V)$ & \(\displaystyle \overline{\soc}_{\g'}^{(r+1)}V_{\left\langle \lambda \right\rangle} = \bigoplus_{m=0}^d \bigoplus_{\substack{|\lambda'| = m-r \\ |\mu'| = d-m-r}} A^{\lambda}_{\lambda', \mu'} V'_{\lambda', \mu'}, \) \\
$V \cong kV' \oplus kV'_*$ & where for $k = 1$ \\
&\(\displaystyle A^{\lambda}_{\lambda', \mu'} = \sum_{\gamma, \delta} c^{\gamma}_{\lambda', \mu'}c^{\lambda}_{\gamma, 2\delta}, \) \\
& and for $k \geq 2$ \\
& \(\displaystyle A^{\lambda}_{\lambda', \mu'} = \sum_{\gamma, \delta, \lambda_1, \mu_1} \sum_{\substack{\beta^+_1, \dots, \beta^+_k \\ \beta^-_1, \dots, \beta^-_k}}
c^{\gamma}_{\lambda_1, \mu_1}c^{\lambda}_{\gamma, 2\delta} C^{(\lambda_1, \mu_1)}_{(\beta^+_1, \dots, \beta^+_k) (\beta^-_1, \dots, \beta^-_k)} D^{(\lambda', \mu')}_{(\beta^+_1, \dots, \beta^+_k) (\beta^-_1, \dots, \beta^-_k)}\) \\
\hline
& \\
$\gl(V', V'_*) \subset \so(V)$ & \(\displaystyle \overline{\soc}_{\g'}^{(r+1)}V_{[\lambda]} = \bigoplus_{m=0}^d \bigoplus_{\substack{|\lambda'| = m-r \\ |\mu'| = d-m-r}} B^{\lambda}_{\lambda', \mu'} V'_{\lambda', \mu'}, \) \\
$V \cong kV' \oplus kV'_*$ & where for $k = 1$ \\
&\(\displaystyle B^{\lambda}_{\lambda', \mu'} = \sum_{\gamma, \delta} c^{\gamma}_{\lambda', \mu'}c^{\lambda}_{\gamma, (2\delta)^T}, \) \\
& and for $k \geq 2$ \\
& \(\displaystyle B^{\lambda}_{\lambda', \mu'} = \sum_{\gamma, \delta, \lambda_1, \mu_1} \sum_{\substack{\beta^+_1, \dots, \beta^+_k \\ \beta^-_1, \dots, \beta^-_k}}
c^{\gamma}_{\lambda_1, \mu_1}c^{\lambda}_{\gamma, (2\delta)^T} C^{(\lambda_1, \mu_1)}_{(\beta^+_1, \dots, \beta^+_k) (\beta^-_1, \dots, \beta^-_k)} D^{(\lambda', \mu')}_{(\beta^+_1, \dots, \beta^+_k) (\beta^-_1, \dots, \beta^-_k)}\) \\
\hline
\end{tabular}
\end{table}
\FloatBarrier

\FloatBarrier
\begin{table}[h]
\centering
\begin{tabular} {|l|l|}
\hline
\multicolumn{2}{|c|}{Embeddings of type III} \\
\hline
& \\
$\so(V') \subset \sp(V)$ & \(\displaystyle \overline{\soc}_{\g'}^{(r+1)} V_{ \left\langle \lambda \right\rangle} \cong \bigoplus_{|\lambda'| = d - 2r} A^{\lambda}_{\lambda'} V'_{[\lambda']}, \) \\
$V \cong 2kV'$ & where for $k = 1$ \\
&\(\displaystyle A^{\lambda}_{\lambda'} = \sum_{\substack{\gamma, \delta, \lambda_1, \mu_1 \\ \alpha, \beta, \sigma, \tau}} c^{\gamma}_{\lambda_1, \mu_1}c^{\lambda}_{\gamma, 2\delta} c^{\lambda'}_{\alpha, \beta} c^{\lambda_1}_{\alpha, 2\sigma} c^{\mu_1}_{\beta, 2\tau}, \) \\
& and for $k \geq 2$ \\
& \(\displaystyle A^{\lambda}_{\lambda'} = \sum
c^{\gamma}_{\lambda_1, \mu_1}c^{\lambda}_{\gamma, 2\delta} C^{(\lambda_1, \mu_1)}_{(\beta^+_1, \dots, \beta^+_k) (\beta^-_1, \dots, \beta^-_k)}D^{(\lambda_2, \mu_2)}_{(\beta^+_1, \dots, \beta^+_k) (\beta^-_1, \dots, \beta^-_k)} c^{\lambda'}_{\alpha, \beta} c^{\lambda_2}_{\alpha, 2\sigma} c^{\mu_2}_{\beta, 2\tau}\) \\
&\\
& (see Section \ref{secFinBran} for the notations)\\
\hline
& \\
$\sp(V') \subset \so(V)$ & \(\displaystyle \overline{\soc}_{\g'}^{(r+1)} V_{[\lambda]} \cong \bigoplus_{|\lambda'| = d - 2r} A^{\lambda}_{\lambda'} V'_{\left\langle \lambda' \right\rangle}, \) \\
$V \cong 2kV'$ & where for $k = 1$ \\
&\(\displaystyle A^{\lambda}_{\lambda'} = \sum_{\substack{\gamma, \delta, \lambda_1, \mu_1 \\ \alpha, \beta, \sigma, \tau}} c^{\gamma}_{\lambda_1, \mu_1}c^{\lambda}_{\gamma, (2\delta)^T} c^{\lambda'}_{\alpha, \beta} c^{\lambda_1}_{\alpha, 2\sigma} c^{\mu_1}_{\beta, (2\tau)^T},\) \\
& and for $k \geq 2$ \\
& \(\displaystyle A^{\lambda}_{\lambda'} = \sum
c^{\gamma}_{\lambda_1, \mu_1}c^{\lambda}_{\gamma, (2\delta)^T} C^{(\lambda_1, \mu_1)}_{(\beta^+_1, \dots, \beta^+_k) (\beta^-_1, \dots, \beta^-_k)} D^{(\lambda_2, \mu_2)}_{(\beta^+_1, \dots, \beta^+_k) (\beta^-_1, \dots, \beta^-_k)} c^{\lambda'}_{\alpha, \beta} c^{\lambda_2}_{\alpha, 2\sigma} c^{\mu_2}_{\beta, (2\tau)^T}\) \\
&\\
\hline
\end{tabular}
\end{table}
\FloatBarrier

\end{document}